\documentclass[reqno]{amsart}

\usepackage{amssymb}
\usepackage{hyperref}
\usepackage{mathtools}
\usepackage[all]{xy}
\usepackage{verbatim}
\usepackage{ifthen}
\usepackage{xargs}
\usepackage{bussproofs}
\usepackage{turnstile}
\usepackage{etex}

\hypersetup{colorlinks=true,linkcolor=blue}

\renewcommand{\turnstile}[6][s]
    {\ifthenelse{\equal{#1}{d}}
        {\sbox{\first}{$\displaystyle{#4}$}
        \sbox{\second}{$\displaystyle{#5}$}}{}
    \ifthenelse{\equal{#1}{t}}
        {\sbox{\first}{$\textstyle{#4}$}
        \sbox{\second}{$\textstyle{#5}$}}{}
    \ifthenelse{\equal{#1}{s}}
        {\sbox{\first}{$\scriptstyle{#4}$}
        \sbox{\second}{$\scriptstyle{#5}$}}{}
    \ifthenelse{\equal{#1}{ss}}
        {\sbox{\first}{$\scriptscriptstyle{#4}$}
        \sbox{\second}{$\scriptscriptstyle{#5}$}}{}
    \setlength{\dashthickness}{0.111ex}
    \setlength{\ddashthickness}{0.35ex}
    \setlength{\leasturnstilewidth}{2em}
    \setlength{\extrawidth}{0.2em}
    \ifthenelse{%
      \equal{#3}{n}}{\setlength{\tinyverdistance}{0ex}}{}
    \ifthenelse{%
      \equal{#3}{s}}{\setlength{\tinyverdistance}{0.5\dashthickness}}{}
    \ifthenelse{%
      \equal{#3}{d}}{\setlength{\tinyverdistance}{0.5\ddashthickness}
        \addtolength{\tinyverdistance}{\dashthickness}}{}
    \ifthenelse{%
      \equal{#3}{t}}{\setlength{\tinyverdistance}{1.5\dashthickness}
        \addtolength{\tinyverdistance}{\ddashthickness}}{}
        \setlength{\verdistance}{0.4ex}
        \settoheight{\lengthvar}{\usebox{\first}}
        \setlength{\raisedown}{-\lengthvar}
        \addtolength{\raisedown}{-\tinyverdistance}
        \addtolength{\raisedown}{-\verdistance}
        \settodepth{\raiseup}{\usebox{\second}}
        \addtolength{\raiseup}{\tinyverdistance}
        \addtolength{\raiseup}{\verdistance}
        \setlength{\lift}{0.8ex}
        \settowidth{\firstwidth}{\usebox{\first}}
        \settowidth{\secondwidth}{\usebox{\second}}
        \ifthenelse{\lengthtest{\firstwidth = 0ex}
            \and
            \lengthtest{\secondwidth = 0ex}}
                {\setlength{\turnstilewidth}{\leasturnstilewidth}}
                {\setlength{\turnstilewidth}{2\extrawidth}
        \ifthenelse{\lengthtest{\firstwidth < \secondwidth}}
            {\addtolength{\turnstilewidth}{\secondwidth}}
            {\addtolength{\turnstilewidth}{\firstwidth}}}
        \ifthenelse{\lengthtest{\turnstilewidth < \leasturnstilewidth}}{\setlength{\turnstilewidth}{\leasturnstilewidth}}{}
    \setlength{\turnstileheight}{1.5ex}
    \sbox{\turnstilebox}
    {\raisebox{\lift}{\ensuremath{
        \makever{#2}{\dashthickness}{\turnstileheight}{\ddashthickness}
        \makehor{#3}{\dashthickness}{\turnstilewidth}{\ddashthickness}
        \hspace{-\turnstilewidth}
        \raisebox{\raisedown}
        {\makebox[\turnstilewidth]{\usebox{\first}}}
            \hspace{-\turnstilewidth}
            \raisebox{\raiseup}
            {\makebox[\turnstilewidth]{\usebox{\second}}}
        \makever{#6}{\dashthickness}{\turnstileheight}{\ddashthickness}}}}
        \mathrel{\usebox{\turnstilebox}}}

\newcommand{\axlabel}[1]{(#1) \phantomsection \label{ax:#1}}

\newcommand{\axref}[1]{(\hyperref[ax:#1]{#1})}

\newcommand{\newref}[4][]{
\ifthenelse{\equal{#1}{}}{\newtheorem{h#2}[hthm]{#4}}{\newtheorem{h#2}{#4}[#1]}
\expandafter\newcommand\csname r#2\endcsname[1]{#3~\ref{#2:##1}}
\expandafter\newcommand\csname R#2\endcsname[1]{#4~\ref{#2:##1}}
\expandafter\newcommand\csname n#2\endcsname[1]{\ref{#2:##1}}
\newenvironmentx{#2}[2][1=,2=]{
\ifthenelse{\equal{##2}{}}{\begin{h#2}}{\begin{h#2}[##2]}
\ifthenelse{\equal{##1}{}}{}{\label{#2:##1}}
}{\end{h#2}}
}

\newref[section]{thm}{Theorem}{Theorem}
\newref{lem}{Lemma}{Lemma}
\newref{prop}{Proposition}{Proposition}
\newref{cor}{Corollary}{Corollary}
\newref{cond}{Condition}{Condition}

\theoremstyle{definition}
\newref{defn}{Definition}{Definition}
\newref{example}{Example}{Example}

\theoremstyle{remark}
\newref{remark}{Remark}{Remark}

\newcommand{\cat}[1]{\mathbf{#1}}
\newcommand{\colim}{\mathrm{colim}}
\newcommand{\C}{\cat{C}}

\newcommand{\Mod}[1]{#1\text{-}\cat{Mod}}
\newcommand{\Th}{\cat{Th}}

\newcommand{\PSt}{\cat{PSt}}
\newcommand{\algtt}{\cat{TT}}

\newcommand{\emptyCtx}{\mathbf{1}}
\newcommand{\nf}{\mathrm{nf}}

\newcommand{\deq}{\equiv}
\newcommand{\repl}{:=}
\newcommand{\type}{}
\newcommand{\Syn}{\mathrm{Syn}}
\newcommand{\Lang}{\mathrm{Lang}}

\newcommand{\Term}{\mathrm{Term}}
\newcommand{\FV}{\mathrm{FV}}

\newcommand{\IdT}{\mathrm{Id}}
\newcommand{\Jeq}{\mathit{Jeq}}
\newcommand{\wUA}{\mathrm{wUA}}
\newcommand{\coeT}{\mathrm{coe}}
\newcommand{\PathT}{\mathrm{Path}}
\newcommand{\transportT}{\mathrm{transport}}

\newcommand{\at}{\mathit{at}}
\newcommand{\unit}{\mathit{unit}}
\newcommand{\Ceq}{\mathit{eq}}
\newcommand{\Id}{\mathit{Id}}

\newcommand{\leftI}{\mathit{left}}
\newcommand{\rightI}{\mathit{right}}

\newcommand{\coe}{\mathit{coe}}
\newcommand{\app}{\mathit{app}}

\newcommand{\refl}{\mathit{refl}}
\newcommand{\subst}{\mathit{subst}}

\newcommand{\transport}{\mathit{transport}}
\newcommand{\ft}{\mathit{ft}}
\newcommand{\ty}{\mathit{ty}}
\newcommand{\ctx}{\mathit{ctx}}
\newcommand{\tm}{\mathit{tm}}

\newcommand{\we}{\mathcal{W}}

\newcommand{\I}{\mathrm{I}}
\newcommand{\J}{\mathrm{J}}
\newcommand{\class}[2]{#1\text{-}\mathrm{#2}}
\newcommand{\Iinj}[1][\I]{\class{#1}{inj}}
\newcommand{\Icell}[1][\I]{\class{#1}{cell}}
\newcommand{\Icof}[1][\I]{\class{#1}{cof}}
\newcommand{\Jinj}[1][]{\Iinj[\J#1]}
\newcommand{\Jcell}[1][]{\Icell[\J#1]}

\newcommand{\cyli}{i}

\numberwithin{figure}{section}

\newcommand{\po}[1][dr]{\save*!/#1+1.2pc/#1:(1,-1)@^{|-}\restore}

\begin{document}

\title{Morita equivalences between algebraic dependent type theories}

\author{Valery Isaev}

\begin{abstract}
We define a notion of equivalence between algebraic dependent type theories which we call Morita equivalence.
This notion has a simple syntactic description and an equivalent description in terms of models of the theories.
The category of models of a type theory often carries a natural structure of a model category.
If this holds for the categories of models of two theories, then a map between them is a Morita equivalence if and only if the adjunction generated by it is a Quillen equivalence.
\end{abstract}

\maketitle

\section{Introduction}

Homotopy type theory can be seen as an internal language of $\infty$-categories.
One way to formalize this point of view is to define a (semi-)model structure on the category of models of a type theory and
prove that it is Quillen equivalent to a model category presenting the $\infty$-category of $\infty$-categories with some additional structure depending on the theory.
This implies that every such $\infty$-category can be presented in the form of a model of this type theory and the theory is naturally ``the internal language'' of its models.
Such (semi-)model structures were constructed in \cite{alg-models} and \cite{kap-lum-model}.
A partial progress on the latter point was made in \cite{kapulkin-szumilo-fin-comp},
where an equivalence between the $\infty$-category of finitely complete $\infty$-categories and the $\infty$-category of models of the type theory with identity types, $\Sigma$-types, and unit types was constructed.

A type theory often can be formulated in several different ways so that the categories of models of these theories are not equivalent
For example, we give several ways to formulate the theory of $\Pi$-types in subsection~\ref{sec:simple}.
A natural question is whether the categories of models of these theories are equivalent in an appropriate sense.
If we can answer this question positively, then it does not matter which theory we use to formulate conjectures about the category of its models such as the one mentioned above.

There is another reason why we might be interested in this question.
There are several theories which should be equivalent in some sense:
\begin{itemize}
\item The theory of a unit type and the theory of a contractible type should be equivalent since the only difference between them is that the former postulate the contractibility of a type judgmentally.
\item It seems that the previous example generalizes to many theories such as the theory of identity types or various theories of inductive types.
We can replace judgmental equality rules with their propositional analogues.
For example, the rule $\Gamma \vdash J(A,a,D,d,a,\refl(a)) \deq d[a]$ is replaced with a new construction $\Jeq(A,D,d,a) : \Id(J(A,a,D,d,a,\refl(a)),d[a])$.
These two theories should be equivalent and, since the theory with the propositional rule is cofibrant, it is a cofibrant replacement of the theory of identity types.
\item If a theory has a judgmental equality between types, then we can replace it with an equivalence between these types.
It is useful to know that these theories are equivalent since there are many examples of models of the theory with the equivalence which are not known to be models of the theory with the judgmental rule.
\item There are two ways in which the theory of $\Sigma$-types can be defined: one of them uses projections and the $\eta$-rule and the other uses usual eliminator rule.
\item The theories of dependent and non-dependent function types should be equivalent (assuming $\Sigma$-types). This is similar to the statement that a category is locally Cartesian closed if and only if it has the $\Pi$-functor.
\item The theory of the interval type defined in \cite{alg-models} should be equivalent to the theory with identity types and the unit type.
\end{itemize}

For every pair of theories listed above, one of the theories can be interpreted in the other, but not the other way around.
This means that these equivalences should be some sort of weak equivalences in a category of type theories.
One definition of such a category was proposed in \cite{alg-tt}.
In this paper, we define several notions of weak equivalences between theories including syntactic equivalence and Morita equivalence.

There is a natural notion of weak equivalences between models of type theories.
It was shown in \cite{alg-models} that if a theory has the interval type, then there is a model structure on the category of models of this theory.
We will prove that there is also a model structure on the category of theories with the interval type with Morita equivalences as weak equivalences.

Morita equivalence between theories $T_1$ and $T_2$ is defined as a map $f : T_1 \to T_2$ such that the unit $\eta_X : X \to f^*(f_!(X))$ of the adjunction $f_! \dashv f^*$ generated by this map is a weak equivalence for every cofibrant object $X$.
Note that the notions of weak equivalences between models and cofibrant models make sense even the model structure does not exist.
If it does exist, then a map is a Morita equivalence if and only if the adjunction is a Quillen equivalence.
This gives us a tool that allows us to compare models of different type theories.
Syntactic equivalences are weaker than Morita equivalences.
A map is a syntactic equivalence if the initial models of theories are weakly equivalent.
More precisely, a map $f : T_1 \to T_2$ is a syntactic equivalence if and only if the unique map $0 \to f^*(0)$ is a weak equivalence.

There is also a characterization of Morita equivalences in syntactic terms.
It seems that this characterization is the most useful one if we want to check that a specific map is a Morita equivalence.
To work with this characterization, it is useful to assume that the theories are confluent.
Roughly speaking, this means that we can choose a direction of axioms so that the relation corresponding to the axioms with the chosen direction is confluent.
We will give a formal definition of confluent theories in the setting of algebraic type theories.

Unfortunately, we still do not know whether all of the examples listed above are indeed Morita equivalences.
Nevertheless, we prove that this is true for the first example and give several other simple examples.

The paper is organized as follows.
In section~\ref{sec:morita}, we give several definitions of syntactic equivalences and Morita equivalences and prove that they are equivalent.
In section~\ref{sec:model-structure}, we construct a model structure on the category of theories with the interval type.
In section~\ref{sec:triv-fib}, we give a characterization of trivial fibrations between theories.
In section~\ref{sec:confluent}, we define confluent theories and prove their properties.
In section~\ref{sec:examples}, we give several examples of Morita equivalences.
In section~\ref{sec:conclusion}, we summarize the results of this paper and discuss issues that prevent us from constructing more examples of Morita equivalences.

\section{Morita equivalences of theories}
\label{sec:morita}

In this section we define several notions of weak equivalence of algebraic dependent type theories.

\subsection{Algebraic dependent type theories}

In this paper, we will need a precise definition of a type theory.
Moreover, we will need a category of such theories.
We will work with definitions given in \cite{alg-tt}.
They are based on the notion of partial Horn theories defined in \cite{PHL}.
Similar ideas were developed by Lumsdaine, Bauer, and Haselwarter \cite{lum-tt}.
In this section, we briefly recall necessary definitions and notations.

A many sorted first-order signature $(\mathcal{S},\mathcal{F},\mathcal{P})$ consists of a set $\mathcal{S}$ of sorts,
a set $\mathcal{F}$ of function symbols and a set $\mathcal{P}$ of predicate symbols.
Each function symbol $\sigma$ is equipped with a signature of the form $\sigma : s_1 \times \ldots \times s_k \to s$, where $s_1$, \ldots $s_k$, $s$ are sorts.
Each predicate symbol $R$ is equipped with a signature of the form $R : s_1 \times \ldots \times s_k$.
If $V$ is an $\mathcal{S}$-set, then the $\mathcal{S}$-set of terms of $T$ with free variables in $V$ will be denoted by $\Term_T(V)$.

An atomic formula is an expression either of the form $t_1 = t_2$ or of the form $R(t_1, \ldots t_n)$,
where $R$ is a predicate symbol and $t_1$, \ldots $t_n$ are terms.
We abbreviate $t = t$ to $t\!\downarrow$.
A Horn formula is an expression of the form $\varphi_1 \land \ldots \land \varphi_n$, where $\varphi_1$, \ldots $\varphi_n$ are atomic formulas.
The conjunction of the empty set of atomic formulas is denoted by $\top$.
A sequent is an expression of the form $\varphi \sststile{}{x_1, \ldots x_n} \psi$, where $x_1$, \ldots $x_n$ are variables
and $\varphi$ and $\psi$ are Horn formulas such that $\FV(\varphi) \cup \FV(\psi) \subseteq \{ x_1, \ldots x_n \}$.
A \emph{partial Horn theory} consists of a signature and a set of Horn sequents in this signature.

An $\mathcal{S}$-set $M$ is a collection of sets $\{ M_s \}_{s \in \mathcal{S}}$.
An interpretation $M$ of a signature $(\mathcal{S},\mathcal{F},\mathcal{P})$ is an $\mathcal{S}$-set $M$
together with a collection of \emph{partial} functions $M(\sigma) : M_{s_1} \times \ldots \times M_{s_k} \to M_s$
for every function symbol $\sigma : s_1 \times \ldots \times s_k \to s$ of $T$
and relations $M(R) \subseteq M_{s_1} \times \ldots \times M_{s_k}$ for every predicate symbol $R : s_1 \times \ldots \times s_k$.
A model of a partial Horn theory $T$ is an interpretation of the underlying signature such that the axioms of $T$ hold in this interpretation.
The category of models of $T$ will be denoted by $\Mod{T}$.

The rules of \emph{partial Horn logic} are listed below.
A \emph{theorem} of a partial Horn theory $T$ is a sequent derivable from $T$ in this logic.
We will write $\varphi \sststile{T}{V} \psi$ to denote the fact that sequent $\varphi \sststile{}{V} \psi$ is derivable in $T$.
\begin{center}
$\varphi \sststile{}{V} \varphi$ \axlabel{b1}
\qquad
\AxiomC{$\varphi \sststile{}{V} \psi$}
\AxiomC{$\psi \sststile{}{V} \chi$}
\RightLabel{\axlabel{b2}}
\BinaryInfC{$\varphi \sststile{}{V} \chi$}
\DisplayProof
\qquad
$\varphi \sststile{}{V} \top$ \axlabel{b3}
\end{center}

\medskip
\begin{center}
$\varphi \land \psi \sststile{}{V} \varphi$ \axlabel{b4}
\qquad
$\varphi \land \psi \sststile{}{V} \psi$ \axlabel{b5}
\qquad
\AxiomC{$\varphi \sststile{}{V} \psi$}
\AxiomC{$\varphi \sststile{}{V} \chi$}
\RightLabel{\axlabel{b6}}
\BinaryInfC{$\varphi \sststile{}{V} \psi \land \chi$}
\DisplayProof
\end{center}

\medskip
\begin{center}
$\sststile{}{x} x\!\downarrow$ \axlabel{a1}
\qquad
$x = y \land \varphi \sststile{}{V,x,y} \varphi[y/x]$ \axlabel{a2}
\end{center}

\medskip
\begin{center}
\AxiomC{$\varphi \sststile{}{V} \psi$}
\RightLabel{, $x \in \FV(\varphi)$ \axlabel{a3}}
\UnaryInfC{$\varphi[t/x] \sststile{}{V,V'} \psi[t/x]$}
\DisplayProof
\end{center}
\medskip

We will give several proofs by induction on the derivation of a sequent.
We need to work with sequents in which the left hand side has some property, but in a derivation of a sequent in this logic the left hand side may vary arbitrary.
Thus we describe another set of rules which is equivalent to this one and in which the left hand side stays the same.
We call these rules the \emph{natural deduction system}.
In this system the right hand side of all sequents is an atomic formula.

\begin{center}
\AxiomC{}
\RightLabel{\axlabel{nv}}
\UnaryInfC{$\varphi \sststile{}{V} x\!\downarrow$}
\DisplayProof
\qquad
\AxiomC{$\varphi \sststile{}{V} a = b$}
\RightLabel{\axlabel{ns}}
\UnaryInfC{$\varphi \sststile{}{V} b = a$}
\DisplayProof
\end{center}

\begin{center}
\AxiomC{}
\RightLabel{\axlabel{nh}}
\UnaryInfC{$\varphi_1 \land \ldots \land \varphi_n \sststile{}{V} \varphi_i$}
\DisplayProof
\qquad
\AxiomC{$\varphi \sststile{}{V} a = b$}
\AxiomC{$\varphi \sststile{}{V} \psi[a/x]$}
\RightLabel{\axlabel{nl}}
\BinaryInfC{$\varphi \sststile{}{V} \psi[b/x]$}
\DisplayProof
\end{center}

\begin{center}
\AxiomC{$\varphi \sststile{}{V} R(t_1, \ldots t_n)$}
\RightLabel{\axlabel{np}}
\UnaryInfC{$\varphi \sststile{}{V} t_i\!\downarrow$}
\DisplayProof
\qquad
\AxiomC{$\varphi \sststile{}{V} \sigma(t_1, \ldots t_n)\!\downarrow$}
\RightLabel{\axlabel{nf}}
\UnaryInfC{$\varphi \sststile{}{V} t_i\!\downarrow$}
\DisplayProof
\end{center}
where $R$ is a predicate symbol of the theory and $\sigma$ is its function symbol.

Finally, for every axiom $\psi_1 \land \ldots \land \psi_n \sststile{}{x_1 : s_1, \ldots x_k : s_k} \chi_1 \land \ldots \land \chi_m$
and for all terms $t_1 : s_1$, \ldots $t_k : s_k$, we have the following rules for all $1 \leq j \leq m$:
\smallskip
\begin{center}
\AxiomC{$\varphi \sststile{}{V} t_i\!\downarrow$, $1 \leq i \leq k$}
\AxiomC{$\varphi \sststile{}{V} \psi_i[t_1/x_1, \ldots t_k/x_k]$, $1 \leq i \leq n$}
\RightLabel{\axlabel{na}}
\BinaryInfC{$\varphi \sststile{}{V} \chi_j[t_1/x_1, \ldots t_k/x_k]$}
\DisplayProof
\end{center}

\begin{prop}
A sequent $\varphi \sststile{}{V} \psi_1 \land \ldots \land \psi_n$ is derivable in the system of rules \axref{b1}-\axref{b6}, \axref{a1}-\axref{a3} if and only if
sequents $\varphi \sststile{}{V} \psi_1$, \ldots $\varphi \sststile{}{V} \psi_n$ are derivable in the natural deduction system.
\end{prop}
\begin{proof}
It is easy to prove the ``if'' part.
Conversely, the rules \axref{b1}, \axref{b4}, and \axref{b5} follow from \axref{nh},
the rules \axref{b3} and \axref{b6} hold trivially,
the rule \axref{a1} follows from \axref{nv},
the rule \axref{a2} follows from \axref{nl} and \axref{nh},
and every axiom is derivable from \axref{na}.

To prove the rule \axref{b2}, we just need to show that if sequents $\varphi \sststile{}{V} \psi_1$, \ldots $\varphi \sststile{}{V} \psi_n$,
and $\psi_1 \land \ldots \land \psi_n \sststile{}{V} \chi$ are derivable in the natural deduction, then $\varphi \sststile{}{V} \chi$ is also derivable.
We can construct a derivation tree for this sequent as a derivation tree for $\psi_1 \land \ldots \land \psi_n \sststile{}{V} \chi$
in which the left hand sides of all sequents are replaced with $\varphi$ and rules \axref{nh} are replaced with derivation trees for $\varphi \sststile{}{V} \psi_i$.

To prove the rule \axref{a3}, consider a derivation tree for a sequent $\varphi \sststile{}{V} \psi$.
To construct a derivation tree for $\varphi[t/x] \sststile{}{V,V'} \psi[t/x]$, we just need to apply the substitution to every sequent in this derivation tree.
The only rule that is not closed under substitution is \axref{nv}.
By assumption, $x \in \FV(\varphi)$.
In this case the sequent $\varphi[t/x] \sststile{}{V,V'} t\!\downarrow$ is derivable from \axref{np}, \axref{nf} and the following rules:
\begin{center}
\AxiomC{$\varphi \sststile{}{V} t_1 = t_2$}
\RightLabel{\axlabel{ne1}}
\UnaryInfC{$\varphi \sststile{}{V} t_1\!\downarrow$}
\DisplayProof
\qquad
\AxiomC{$\varphi \sststile{}{V} t_1 = t_2$}
\RightLabel{\axlabel{ne2}}
\UnaryInfC{$\varphi \sststile{}{V} t_2\!\downarrow$}
\DisplayProof
\end{center}
The rule \axref{ne2} follows from \axref{nl} if we take $\psi(x) = (x = b)$.
The rule \axref{ne1} follows from \axref{ne2} and \axref{ns}.
\end{proof}

We will need the following lemma later:

\begin{lem}[mcf]
A sequent $\varphi \sststile{}{x_1, \ldots x_n} \psi$ is provable in a theory $T$ if and only if
the sequent $\sststile{}{} \psi[c_1/x_1, \ldots c_n/x_n]$ is provable in the theory $T \cup \{ \sststile{}{} c_i\!\downarrow\ |\ 1 \leq i \leq n \} \cup \{ \varphi[c_1/x_1, \ldots c_n/x_n] \}$,
where $c_1$, \ldots $c_n$ are fresh constants.
\end{lem}
\begin{proof}
This follows from \cite[Theorem~10, Theorem~11]{PHL}.
\end{proof}

Partial Horn theories with a given set of sorts $\mathcal{S}$ form a category which we will denote by $\Th_\mathcal{S}$ (see \cite[Section~2]{alg-tt} for a definition).
The category of models $\Mod{T}$ embeds fully faithfully into the category $T/\Th_\mathcal{S}$.
We will denote the embedding functor by $\Lang$.
This functor has a right adjoint $\Syn : T/\Th_\mathcal{S} \to \Mod{T}$.
The model $\Syn(T')$ is defined as the syntactic model of the theory $T'$.
The precise definition of these functors is given in \cite[Section~4.1]{alg-models}.

The \emph{theory of substitutions} is the theory with $\mathcal{S} = \{ \ctx, \tm \} \times \mathbb{N}$ as the set of sorts, function symbols given below, and axioms listed in \cite[Section~3.1]{alg-tt}.
\begin{align*}
\emptyCtx      & : (\ctx,0) \\
\ft_n          & : (\ty,n) \to (\ctx,n) \\
\ty_n          & : (\tm,n) \to (\ty,n) \\
v_{n,i}        & : (\ctx,n) \to (\tm,n) \text{, } 0 \leq i < n \\
\subst_{p,n,k} & : (\ctx,n) \times (p,k) \times (\tm,n)^k \to (p,n) \text{, } p \in \{ \tm, \ty \}
\end{align*}

We often omit index $n$ in these function symbols.
So, we will write $\ft$, $\ty$, and $v_i$ instead of $\ft_n$, $\ty_n$, and $v_{n,i}$, respectively.
Function symbols $v_{n,i}$ represent de Bruijn indices.
We will sometimes use named representation of terms.
So, $A_1, \ldots A_n \vdash B(v_{n-1, \ldots v_0})$ and $x_1 : A_1, \ldots x_n : A_n \vdash B(x_1, \ldots x_n)$ represent the same judgment.
The judgment itself simply denotes the formula $\ft(B(v_{n-1, \ldots v_0})) = A_n \land \ft(A_n) = A_{n-1} \land \ldots \land \ft(A_1) = A_0$.
Similarly, judgement $\Gamma \vdash b : B$ represents formula $\ty(b) = B \land (\Gamma \vdash B)$.

Note that $\subst_{p,n,k}$ denotes the usual type-theoretic substitution while $t[s/x]$ denotes substitution on the meta level of algebraic theories.

We let $\ft^i_n : (\ctx,n+i) \to (\ctx,n)$ and $\ctx_{p,n} : (p,n) \to (\ctx,n)$ be the following derived operations:
\begin{align*}
\ft^0_n(A)      & = A \\
\ft^{i+1}_n(A)  & = \ft^i_n(\ft_{n+i}(A)) \\
\ctx_{\ty,n}(t) & = \ft_n(t) \\
\ctx_{\tm,n}(t) & = \ft_n(\ty_n(t)) \\
\ctx^i_{p,n}(t) & = \ft^i_n(\ctx_{p,n+i}(t))
\end{align*}
We also write $(\ty,n)$ for $(\ctx,n+1)$.

Let $\mathcal{F}_0$ be a set of function symbols and let $\mathcal{P}_0$ be a set of predicate symbols.
We call elements of these sets basic function symbols and basic predicate symbols, respectively.
Then we define the full sets of function and predicate symbols:
\begin{align*}
\mathcal{F} = \{ & \sigma_m : (\ctx,m) \times (p_1,m+n_1) \times \ldots \times (p_k,m+n_k) \to (p,m+n) \mid \\
                 & m \in \mathbb{N}, \sigma \in \mathcal{F}_0, \sigma : (p_1,n_1) \times \ldots \times (p_k,n_k) \to (p,n) \} \\
\mathcal{P} = \{ & R_m : (\ctx,m) \times (p_1,m+n_1) \times \ldots \times (p_k,m+n_k) \mid \\
                 & m \in \mathbb{N}, R \in \mathcal{P}_0, R : (p_1,n_1) \times \ldots \times (p_k,n_k) \}
\end{align*}

An \emph{algebraic dependent type theory} is a theory of the form $(\mathcal{S}, \mathcal{F}_s \cup \mathcal{F}, \mathcal{P}, \mathcal{A}_s \cup \mathcal{A})$, where $\mathcal{S}$, $\mathcal{F}$, and $\mathcal{P}$ are defined above,
$\mathcal{F}_s$ is the set of function symbols of the theory of substitutions, $\mathcal{A}_s$ is the set of its axioms, and $\mathcal{A}$ is an arbitrary set of axioms such that the following sequents are derivable for every $\sigma_m \in \mathcal{F}$ and $R_m \in \mathcal{P}$:
\begin{align*}
\sigma_m(\Gamma, x_1, \ldots x_k)\!\downarrow\ & \sststile{}{\Gamma, x_1, \ldots x_k} \ctx^n_{p,m}(\sigma_m(\Gamma, x_1, \ldots x_k)) = \Gamma \\
\sigma_m(\Gamma, x_1, \ldots x_k)\!\downarrow\ & \sststile{}{\Gamma, x_1, \ldots x_k} \bigwedge_{1 \leq i \leq k} \ctx^{n_i}_{p_i,m}(x_i) = \Gamma \\
R_m(\Gamma, x_1, \ldots x_k) & \sststile{}{\Gamma, x_1, \ldots x_k} \bigwedge_{1 \leq i \leq k} \ctx^{n_i}_{p_i,m}(x_i) = \Gamma
\end{align*}
Moreover, such a theory must satisfy the condition given in \cite[Definition~4.5]{alg-tt}, which just says that $\subst$ commutes with all function symbols.

We will also use the theory defined in \cite{alg-models}, which we denote by $\coeT_1 + \sigma + \PathT + \wUA$.
This theory has an interval type $I$ with two constructors $left$ and $right$ and an eliminator $\coe$, which is just the eliminator for the unit type.
It also has the following axiom, which we denote by $\sigma$:
\[ \coe(x.A, a, i) = a \text{, if } x \notin \FV(A) \]
It also has the type of paths, which we will denote by $\Id(a,a')$ and a weak univalence axiom.

\subsection{Model categories of models of type theories}

To define Morita equivalences between two theories $T_1$ and $T_2$, they must have some additional structure.
We assume that all of the theories are equipped with a morphism from the theory that has one function symbol
$\Id : (\tm,0) \times (\tm,0) \to (\ty,0)$ and the only axiom $\Id(x,y)\!\downarrow\ \sststile{}{x,y} \ty(x) = \ty(y)$.
We will denote this theory by $\IdT_0$.
We often need to assume even more structure, but we will always state additional assumptions explicitly.

Let $T$ be a theory under $\IdT_0$ and let $X$ be a model of $T$.
A \emph{relative homotopy} between terms $a,a' \in X_{(\tm,n)}$ is a term $h \in X_{(\tm,n)}$ such that $\ty(h) = \Id(a,a')$.
A \emph{relative homotopy} between types $A,A' \in X_{(\ty,n)}$ is a tuple $(f,g,p,g',p')$, where $f,g,p,g',p' \in X_{(\tm,n+1)}$ such that
\begin{align*}
x : A & \vdash f : A' \\
y : A' & \vdash g : A \\
x : A & \vdash p : \Id(g[y \mapsto f], x) \\
y : A' & \vdash g' : A \\
y : A' & \vdash p' : \Id(f[x \mapsto g], y)
\end{align*}

In general the homotopy relation is not an equivalence relation, but it is if $T$ also has the reflexivity and transport operations:
\begin{center}
\AxiomC{}
\UnaryInfC{$\vdash \refl(x) : \Id(x,x)$}
\DisplayProof
\qquad
\AxiomC{$\vdash p : \Id(a,a')$}
\AxiomC{$\vdash b : B[a]$}
\BinaryInfC{$\vdash \transport(B,a,a',p,b) : B[a']$}
\DisplayProof
\end{center}

Let $X$ and $Y$ be models of a theory with identity types.
A morphism of models $f : X \to Y$ is \emph{weak equivalence} if it satisfies the following conditions:
\begin{enumerate}
\item For all $A \in X_{(\ty,n)}$ and $a \in Y_{(\tm,n)}$ such that $\ty(a) = f(A)$,
there is a term $a' \in X_{(\tm,n)}$ such that $\ty(a') = A$ and $f(a')$ is relatively homotopic to $a$.
In this case we will say that $f$ is \emph{essentially surjective on terms}.
\item For all $\Gamma \in X_{(\ctx,n)}$ and $A \in Y_{(\ty,n)}$ such that $\ft(A) = f(\Gamma)$,
there is a type $A' \in X_{(\ty,n)}$ such that $\ft(A') = \Gamma$ and $f(A')$ is relatively homotopic to $A$.
In this case we will say that $f$ is \emph{essentially surjective on types}.
\end{enumerate}

For every theory $T$ under $\IdT_0$, we define a set $\I$ of maps in the category of models of $T$ as the set consisting of maps of the form
\[ F(\{ A : (d_p,n) \}) \to F(\{ e_p(a) = A \}) \]
where $d_\ty = \ctx$, $d_\tm = \ty$, $e_\ty(a) = \ft(a)$, $e_\tm(a) = \ty(a)$,
and $F(S)$ is the free model generated by the specified generators and relations.
The class of \emph{cofibrations} of $\Mod{T}$ is generated by $\I$.

Let $\J$ be the set consisting of maps of the following forms:
\begin{align*}
F(\{ a : (\tm,n) \}) & \to F(\{ a, a' : (\tm,n), p : \Id(a,a') \}) \\
F(\{ A : (\ty,n) \}) & \to F(\{ A, A' : (\ty,n), f,g,p,g',p' : (\tm,n+1), S \}),
\end{align*}
where $S$ is the set of formulas asserting that $(f,g,p,g',p')$ is a relative homotopy between $A$ and $A'$.
The class of \emph{anodyne extensions} is generated by $\J$.

We are interested in question when the classes of cofibrations and weak equivalences as defined above determine a model structure or a left semi-model structure.
We will use the definition of left semi-model structures given in \cite[Lemma~6.7]{kap-lum-model}.
We will say that a theory is \emph{a model theory} (resp., \emph{a semi-model theory}) if this model structure (resp., left semi-model structure) exists on the category of its models.
We proved several results about model structures in \cite{f-model-structures} which are useful when working with this model structure and they also apply to left semi-model structure.

\begin{prop}[model-theories]
A theory is a model theory if and only if the weak equivalences satisfy the 2-out-of-3 property and pushouts of maps in $\J$ are weak equivalences.
A theory is a semi-model theory if and only if the weak equivalences satisfy the 2-out-of-6 property and a pushout of a map in $\J$ is a weak equivalence if it has a cofibrant domain.
\end{prop}
\begin{proof}
This follows from \cite[Proposition~3.1]{f-model-structures} and the fact that weak equivalences are closed under transfinite compositions.
\end{proof}

It was shown in \cite{kap-lum-model} that a certain theory with identity types, $\Sigma$-types, and $\Pi$-types is a semi-model theory.
We proved in \cite{alg-models} that all theories under $\coeT_1 + \sigma + \PathT + \wUA$ are model theories.
The argument that shows this actually applies to any theory under $\coeT^{l'}_2 + \PathT + \wUA$ (see the cited paper for the definition of these theories).
We will prove that a theory under $\coeT_1 + \sigma + \PathT + \wUA$ is often equivalent to a theory under $\coeT^{l'}_2 + \PathT + \wUA$, so we might work with either of them,
but we prefer to use the latter theory since it is harder to show that theories with the $\sigma$-rule are confluent (see section~\ref{sec:confluent} for a definition of a confluent theory).

\subsection{Morita equivalences}

Now, we can give the main definition of this paper.

\begin{defn}
A \emph{Morita equivalence} between theories $T_1$ and $T_2$ is a morphism $f : T_1 \to T_2$ such that for every cofibrant model $X$ of $T_1$,
the unit $\eta_X : X \to f^*(f_!(X))$ of the adjunction $f_! \dashv f^*$ is a weak equivalence.
We will say that $f$ is a \emph{strict Morita equivalence} if $\eta_X$ is a weak equivalence for every $X$.
We will say that $f$ is a \emph{syntactic equivalence} if $\eta_X$ is a weak equivalence when $X$ is the initial model.
\end{defn}

If the theories are semi-model, then we can give a characterization of Morita equivalences in terms of the semi-model structures on the categories of their models.

\begin{prop}[morita-quillen]
Let $T_1$ and $T_2$ be semi-model theories.
Then, for every morphism $f : T_1 \to T_2$, the adjunction $f_! \dashv f^*$ is a Quillen adjunction.
It is a Quillen equivalence if and only if $f$ is a Morita equivalence.
\end{prop}
\begin{proof}
Since $f_!$ is a left adjoint, it preserves object defined by generators and relations.
Since the set of generating cofibration $\I$ and the set of generating trivial cofibration $\J_\I$
are both defined in terms of generators and relations, this implies that $f_!$ preserves them.
Hence $f_! \dashv f^*$ is a Quillen adjunction.
The second part of the proposition follows from \cite[Corollary~3.9]{f-model-structures}.
\end{proof}

We can give a useful characterization of (strict) Morita equivalences.
To do this, we need to define a notion of a relative homotopy between terms in a theory.
Let $T$ be a theory with identity types and let $\varphi$ be a formula of $T$.
A \emph{relative homotopy} between types $A,A' \in \Term_T(V)_{(\ty,n)}$ with respect to $\varphi$ is a tuple $f,g,p,g',p' \in \Term_T(V)_{(\tm,n+1)}$
such that sequent $\varphi \sststile{}{V} \psi$ is derivable in $T$, where $\psi$ is the conjunction of formulas that appear in the definition of a relative homotopy for models.
If $a,a' \in \Term_T(V)_{(\tm,n)}$ are terms such that $\varphi \sststile{}{V} \ty(a) = \ty(a')$, then a \emph{relative homotopy} between $a$ and $a'$ with respect to $\varphi$ is a term $h \in \Term_T(V)_{(\tm,n)}$
such that sequent $\varphi \sststile{}{V} \ty(h) = \Id(a,a')$ is derivable in $T$.
If $a$ and $a'$ are such that only $\varphi \sststile{}{V} \ft(\ty(a)) = \ft(\ty(a'))$ is true, then a \emph{relative (heterogeneous) homotopy} between $a$ and $a'$ with respect to $\varphi$
is a relative homotopy $f,g,p,g',p'$ between $\ty(a)$ and $\ty(a')$ together with a relative homotopy between $f[a]$ and $a'$.

We can study lifting properties of maps $\eta_X : X \to f^*(f_!(X))$ in syntactical terms.
Let $V$ be a set of variables and let $\varphi$ be a formula with free variables in $V$.
Then we can consider the syntactic model $M$ corresponding to the pair $(V,\varphi)$.
It is the model generated by constants corresponding to variables in $V$ together a relation corresponding to $\varphi$ (for a precise definition, see the proof of \rprop{str-morita-char}).
Then the lifting properties of $\eta_M$ are closely related to certain lifting properties in the category of theories.
More precisely, we will show that $\eta_M$ has a lifting property for all $M$ if and only if a certain lifting property holds in the category of theories.

Let us describe this lifting property.
Let $V$ be a set of variables and let $\varphi$ be a formula with free variables in $V$.
We will say that a morphism $f : T_1 \to T_2$ of theories with identity types has \emph{the weak lifting property} with respect to $V,\varphi$ if
for every term $A \in \Term_{T_1}(V)_{(d_p,n)}$ and every term $a \in \Term_{T_2}(V)_{(p,n)}$ such that $\varphi \sststile{}{V} A\!\downarrow$ and $f(\varphi) \sststile{}{V} e_p(a) = f(A)$,
there exists a term $a' \in \Term_{T_1}(V)_{(p,n)}$ such that $f(a')$ is relatively homotopic to $a$ with respect to $\varphi$.
We will say that $f$ has \emph{the lifting property} with respect to $V,\varphi$ if $f(a')$ is not only homotopic to $a$, but actually is equal to it.

If $P$ is a set of pairs of the form $V,\varphi$, then we will say that a map has the (weak) lifting property with respect to $P$ if it has this property with respect to every element of $P$.
We define $P_0$ as the singleton set $\{ \varnothing,\top \}$, $P_S$ as the set of all pairs, and $P_M$ as the set of pairs $V,\varphi$ such that $V = \{ x_1, \ldots x_k \}$
and $\varphi = \varphi_1 \land \ldots \land \varphi_k$, where $\varphi_i$ equals to $e_p(x_i) = t_i$,
where $t_i$ is a term of $T_1$ with free variables in $\{ x_1, \ldots x_{i-1} \}$ such that for every $1 \leq i \leq k$,
sequent $\varphi_1 \land \ldots \land \varphi_{i-1} \sststile{}{x_1, \ldots x_{i-1}} t_i\!\downarrow$ is derivable in $T_1$.

\begin{prop}[str-morita-char]
A morphism $f : T_1 \to T_2$ between theories with identity types is a strict Morita equivalence if and only if it has the weak lifting property with respect to $P_S$.
\end{prop}
\begin{proof}
First, we need to introduce an auxiliary construction.
Let $T$ be a theory, let $V$ be a set of variables, and let $\mathcal{A}$ be a set formulas of $T$ with variables in $V$.
Then we define $\Syn(T,V,\mathcal{A})$ as $\Syn(T \cup \{ O_x : s\ |\ x \in V_s \} \cup \mathrm{sp}(\mathcal{A}))$ (functors $\Syn$ and $\Lang$ are defined in \cite{alg-models}),
where $\mathrm{sp}(\mathcal{A})$ consists of formulas of the form $\sststile{}{} O_x\!\downarrow$ for every $x \in V$
and formulas of $\mathcal{A}$ in which every variable $x$ is replaced with $O_x$.
If $f : T_1 \to T_2$ is a morphism of theories, then it is easy to see that $f_!(\Syn(T_1,V,\mathcal{A})) = \Syn(T_2,V,f(\mathcal{A}))$.

Let us prove the ``only if'' direction.
Note that elements of $\Syn(T_1, V, \{\,\sststile{}{}~\varphi\,\})$ correspond to terms $t$ of $T_1$ with variables in $V$ such that $\varphi \sststile{T_1}{V} t\!\downarrow$.
Moreover, two terms $t_1$ and $t_2$ map to the same element under this correspondence if and only if $\varphi \sststile{T_1}{V} t_1 = t_2$.
An analogous statement holds for $\Syn(T_2, V, \{\,\sststile{}{}~f(\varphi)\,\})$.
Using this correspondence, the required conditions immediately follow from the fact that
map $\Syn(T_1, V, \{\,\sststile{}{}~\varphi\,\}) \to f^*(\Syn(T_2, V, \{\,\sststile{}{}~f(\varphi)\,\}))$ is a weak equivalence.

Now, let us prove the ``if'' direction.
Let $M$ be a model of $T_1$.
Note that $M$ is isomorphic to $\Syn(T_1, U(M), \mathcal{A})$, where $U(M)$ is the underlying set of $M$ and $\mathcal{A}$ is the set of formulas of the form
$x = \sigma(x_1, \ldots x_k)$ and $R(x_1, \ldots x_k)$ for all $x, x_1, \ldots x_k \in M$ such that these formulas hold in $M$.
Note that $sp(\mathcal{A})$ is the set of axioms of $\Lang(M)$.

Let $A \in M_{(d_p,n)}$ and $a \in f^*(f_!(M))$ be elements such that $e_p(a) = A$.
Since $f_!(M) = \Syn(T_2, U(M), f(\mathcal{A}))$, $a$ is a closed term of $T_2$.
There is a finite subset $\mathcal{A}_0$ of $\mathcal{A}$ such that $\sststile{T_2 \cup \mathrm{sp}(\mathcal{A}_0)}{} e_p(a) = A$.
Let $\varphi$ be the conjunction of $\mathcal{A}_0$, and let $b$ and $B$ be $a$ and $A$, respectively, in which every constant $O_x$ is replaced with variable $x$.
Then $\varphi \sststile{T_2}{U(M)} e_p(b) = B$.
By assumption, there exist a term $b' \in \Term_{T_1}(U(M))_{(p,n)}$ and a relative homotopy $h$ between $f(b)$ and $b'$.
These terms correspond under $\mathrm{sp}$ to elements of $M$ and $f^*(f_!(M))$, respectively.
These conditions imply that $b'$ is the required lifting and $h$ is the required homotopy.
\end{proof}

Analogous characterizations hold for Morita and syntactic equivalences:

\begin{prop}[morita-char]
A morphism $f : T_1 \to T_2$ between theories with identity types is a Morita equivalence if and only if it has the weak lifting property with respect to $P_M$.
\end{prop}
\begin{proof}
Suppose that $f$ is a Morita equivalence.
To prove that $f$ has the weak lifting property, we just need to show that model $M = \Syn(T_1, \{ x_1, \ldots x_k \}, \{\,\sststile{}{}~\varphi\,\})$
constructed in the previous proposition is cofibrant.
Note that for every $1 \leq i \leq k$, we have the following pushout square:
\[ \xymatrix{ F(\{ A : (d_p,n) \}) \ar[d] \ar[r] &     \Syn(T_1, \{ x_1, \ldots x_{i-1} \}, \{\,\sststile{}{} \varphi_1 \land \ldots \land \varphi_{i-1} \,\}) \ar[d] \\
              F(\{ e_p(a) = A \})         \ar[r] & \po \Syn(T_1, \{ x_1, \ldots x_i \}, \{\,\sststile{}{} \varphi_1 \land \ldots \land \varphi_i \,\}),
            } \]
where the top arrow maps $A$ to $t_i$ and the bottom arrow maps $a$ to $x_i$.
This shows that $M$ is a relative $\I$-cell complex.

Now, let us prove the converse.
We just need to show that if $M$ is a cofibrant model of $T_1$, then we can choose formula $\varphi$
in the second part of the proof of the previous proposition so that it satisfies the conditions of this proposition.

Since every cofibrant object is a retract of a relative $\I$-cell complex and Morita equivalences are closed under retracts, we may assume that $M$ is a relative $\I$-cell complex.
Moreover, we may assume that there are subsets $\{S_i\}_{i \in \mathbb{N}}$ of elements of $M$ such that we have the following pushout diagrams:
\[ \xymatrix{ \coprod_{x \in S_i} F(\{ A_x : (d_p,n) \}) \ar[d] \ar[r] &     M_i \ar[d] \\
              \coprod_{x \in S_i} F(\{ e_p(a_x) = A_x \})       \ar[r] & \po M_{i+1},
            } \]
$M_0$ is the initial model, $M$ is the colimit of $M_i$, and map $F(\{ e_p(a_x) = A_x \}) \to M_{i+1} \to M$ sends $a_x$ to $x$.

Note that $M_i$ is isomorphic to $\Syn(T_1, \bigcup_{1 \leq j \leq i} S_j, \mathcal{A}_i)$,
where $\mathcal{A}_i$ consists of formulas of the form $e_p(x) = t$, where $x \in S_i$ and $t \in \Term_{T_1}(\bigcup_{1 \leq j < i} S_j)$ corresponds to the image of $A_x$ in $M_{i-1}$.
Thus, $M$ is isomorphic to $\Syn(T_1, \bigcup_{i \in \mathbb{N}} S_i, \bigcup_{i \in \mathbb{N}} \mathcal{A}_i)$.
Now, if we choose a finite subset of $\bigcup_{i \in \mathbb{N}} \mathcal{A}_i$ as before, then the conjunction of this subset satisfies the required conditions.
\end{proof}

\begin{prop}[syn-equiv-char]
A morphism $f : T_1 \to T_2$ between theories with identity types is a syntactic equivalence if and only if it has the weak lifting property with respect to $P_0$.
\end{prop}
\begin{proof}
This is obvious since elements of the initial model of $T_1$ are closed terms $t$ of $T_1$ such that $\sststile{}{} t\!\downarrow$ is derivable.
\end{proof}

We will show that there is a model structure on the category of theories with the interval type, path types and the weak univalence axiom as described in \cite{alg-models}.
Note that if we assume only usual identity types with the J rule since, then no such model structure (or left semi-model structure, or structure of a cofibration category) can exist since trivial cofibrations are not closed under pushouts.
Indeed, let $\IdT$ be any version of the theory of identity types, let $T_1 = \IdT \amalg \{ A : (\ty,0), A\!\downarrow \}$, and let $T_2$ be $\IdT$ together with two constants $A,A' : (\ty,0)$ such that $\sststile{T_2}{} A\!\downarrow \land A'\!\downarrow$ and an equivalence between $A$ and $A'$.
\Rprop{str-morita-char} implies that the obvious morphism $T_1 \to T_2$ is a strict Morita equivalence.
Now, consider the theory $T_3 = \{ \sigma : (\ty,0) \to (\ty,0), \sigma(x)\!\downarrow \}$.
Then the map $T_1 \amalg T_3 \to T_2 \amalg T_3$ is not even a syntactic equivalence since types $\sigma(A)$ and $\sigma(A')$ are equal in $T_2 \amalg T_3$, but there is no term between them in $T_1 \amalg T_3$.

It was shown in \cite{alg-models} that the category of models of a theory under $\coeT_1 + \sigma + \PathT + \wUA$ carries a model structure.
If the theory has only identity types, then there is only a left semi-model structure as shown in \cite{kap-lum-model}.
We can generalize this theorem using the following lemma:

\begin{lem}
Let $T_1$ be a theory such that the weak equivalences in $\Mod{T_1}$ satisfy the 2-out-of-6 property.
If $T_2$ is a semi-model theory and $F : T_1 \to T_2$ is a Morita equivalence, then $T_1$ is also semi-model and $F_! \dashv F^*$ is a Quillen equivalence between $\Mod{T_1}$ and $\Mod{T_2}$.
\end{lem}
\begin{proof}
By \rprop{model-theories}, we just need to prove that pushouts of maps in $\J$ with cofibrant codomains are weak equivalences in $\Mod{T_1}$.
Let $f : X \to Y$ be a pushout of a map in $\J$ such that $X$ is cofibrant.
Since $F_!$ preserves pushouts and maps in $\J$, the map $F_!(f)$ is a weak equivalence.
The functor $F^*$ always preserves weak equivalences.
Thus, $F^*(F_!(f))$ is a weak equivalence.
Since $X$ and $Y$ are cofibrant, the maps $\eta_X : X \to F^*(F_!(X))$ and $\eta_Y : Y \to F^*(F_!(Y))$ are weak equivalences.
Hence, $f$ is also a weak equivalence.
\end{proof}

Note that \cite[Proposition~3.3]{kap-lum-model} implies that, for all theories with identity types, $\Sigma$-types, and the unit type, the weak equivalences satisfy the 2-out-of-6 property.
Thus, the first condition of the previous lemma is often true.
We believe that this might be true more generally for all theories with only identity types, but the proofs become much harder without $\Sigma$-types.

Finally, let us prove an analogous lemma for strict Morita equivalences:

\begin{lem}
Let $T_1$ be a theory under $\IdT_0 + \transportT$.
If $T_2$ is a model theory and $F : T_1 \to T_2$ is a strict Morita equivalence, then $T_1$ is also model and $F_! \dashv F^*$ is a Quillen equivalence between $\Mod{T_1}$ and $\Mod{T_2}$.
\end{lem}
\begin{proof}
Since we have the transport operation, the homotopy relation is transitive.
This implies that weak equivalences are closed under composition.
It is also easy to see that if $f : X \to Y$ and $g : Y \to Z$ are maps such that $g$ and $g \circ f$ are weak equivalences, then $f$ is also a weak equivalence.
Now, the same proof as in the previous lemma shows that $F_!$ reflects weak equivalences.

By Theorem~4.2, Proposition~4.3, and Proposition~4.4 from \cite{f-model-structures}, the model structure on $\Mod{T_1}$ exists if there is a path object functor $P : \Mod{T_1} \to \Mod{T_1}$
such that $p : P(X) \to X \times X$ belongs to $\Jinj$ and $\pi_1 \circ p$ belongs to $\Iinj$.
We can define $P(X)$ as usual factorization of the diagonal $X \to X \times X$ into a map $t : X \to P(X)$ in $\Jcell$ followed by a map $p : P(X) \to X \times X$ in $\Jinj$.
Since $F_!$ preserves maps in $\Jcell$, the map $F_!(t)$ is a weak equivalence.
By the 2-out-of-3 property, the map $F_!(\pi_1 \circ p)$ is also a weak equivalence.
Since $F_!$ reflects weak equivalences, this implies that $\pi_1 \circ p$ is a weak equivalence.
Now, since $\pi_1 \circ p$ belongs to $\Jinj$, \cite[Proposition~3.1]{f-model-structures} implies that it also belongs to $\Iinj$.
\end{proof}

\section{Model structure on theories}
\label{sec:model-structure}

In this section we define a model structure on the category of algebraic dependent type theories with enough structure.
Weak equivalences in this model structure are precisely Morita equivalences.
This model structure can be used to prove that some map is a Morita equivalence.
For example, we can use the 2-out-of-3 property.
We can also use the fact that every weak equivalence factors into a trivial cofibration and a trivial fibration and it is easier to check that a map is a trivial cofibration or a trivial fibration.
Several characterizations of trivial fibrations will be given in section~\ref{sec:triv-fib}.
To prove that a map is a trivial cofibration, we can construct a homotopy inverse for it.
Since all theories are fibrant in this model structures, such an inverse always exists.

\subsection{Categories of theories}

It was shown in \cite{PHL} that partial Horn theories are equivalent to essentially algebraic theories.
It follows that categories of models of these theories are locally presentable.
In this subsection we will prove that different categories of theories are also locally finitely presentable.

We will consider a prestable theory $T$ under some prestable theory $B$.
Recall that a prestable theory is a theory $T$ with a map $\alpha : L(T) \to T$, where $L$ is a functor defined in \cite{alg-tt}.
It was shown in \cite[Lemma~4.4]{alg-tt} that every such theory is isomorphic to a contextual theory,
that is a theory which has $\mathcal{F}_B \amalg (\mathcal{F}_0 \times \mathbb{N})$,
$\mathcal{P}_B \amalg (\mathcal{P}_0 \times \mathbb{N})$ and $\mathcal{A}_B \amalg \mathcal{A}_0$ as the sets of function and predicate symbols and the set of axioms, respectively,
where $\mathcal{F}_0$, $\mathcal{P}_0$, and $\mathcal{A}_0$ are some sets and $\mathcal{F}_B$, $\mathcal{P}_B$, and $\mathcal{A}_B$ are the corresponding sets of $B$.
Elements of $\mathcal{F}_0$, $\mathcal{P}_0$ and $\mathcal{A}_0$ are called basic function symbols, basic predicate symbols, and basic axioms.

Now, we give an explicit construction of coproducts and coequalizers in the category $B/\PSt_{\mathcal{S}_0}$ of prestable theories under $B$,
which is similar to the one described in \cite[Proposition~2.12]{alg-tt} for the category of theories.
If $\{ T_i \}_{i \in I}$ is a set of theories under $B$, then the basic function and predicate symbols
and axioms of $\coprod_{i \in I} T_i$ are the disjoint union of corresponding sets of $T_i$.
If $f,g : T \to T'$ is a pair of maps of theories under $B$, then their coequalizer can be defined as
$T'$ together with the following axioms for every basic function symbol $\sigma$ and every basic predicate symbol $R$ of $T$:
\begin{align*}
& \sststile{}{x_1, \ldots x_k} f(\sigma(x_1, \ldots x_k)) \cong g(\sigma(x_1, \ldots x_k)) \\
& f(R(x_1, \ldots x_k)) \ssststile{}{x_1, \ldots x_k} g(R(x_1, \ldots x_k))
\end{align*}

The colimit of a diagram $T : I \to B/\PSt_{\mathcal{S}_0}$ can be described as the coequalizer of the coproduct $\coprod_{i \in I} T_i$ as usual.
Thus we can assume that the sets of basic function and predicate symbols of $\colim_{i \in I} T_i$ are disjoint unions of the corresponding sets of $T_i$.
The axioms of $\colim_{i \in I} T_i$ are axioms of $T_i$ together with axioms of the form $\sststile{}{x_1, \ldots x_n} \sigma(x_1, \ldots x_n) \cong f(\sigma(x_1, \ldots x_n))$
and $R(x_1, \ldots x_n) \ssststile{}{x_1, \ldots x_n} f(R(x_1, \ldots x_n))$ for every morphism $f : T_i \to T_j$
in the diagram and every function symbol $\sigma$ and predicate symbol $R$ of $T_i$ which are not symbols of $B$.

Let $\lambda$ be a regular cardinal.
We will say that a theory $T = ((\mathcal{S}, \mathcal{F}_0 \amalg \mathcal{F}, \mathcal{P}_0 \amalg \mathcal{P}), \mathcal{A}_0 \amalg \mathcal{A})$
in $\Th_B$ is \emph{$\lambda$-small} if cardinalities of sets $\mathcal{F}$, $\mathcal{P}$ and $\mathcal{A}$ are less than $\lambda$.
We will say that $T$ is \emph{finite} if it is $\aleph_0$-small.

\begin{prop}[theories-presentable]
The category of prestable theories under a prestable theory $B$ is locally finitely presentable.
An object of this category is $\lambda$-presentable if and only if it is isomorphic to a $\lambda$-small object.
\end{prop}
\begin{proof}
First, let us prove that every $\lambda$-small object is $\lambda$-presentable.
Let $\colim_{i \in I} T_i$ be a directed colimit of theories in $B/\PSt_{\mathcal{S}_0}$.
Every term and every formula of a theory is constructed from a finite number of function and predicate symbols.
Thus for every formula of $\colim_{i \in I} T_i$ there exists a theory $T_i$ such that this formula belongs to $T_i$.
The same is true for terms and restricted terms.

Every derivation of a theorem $\varphi \sststile{}{V} \psi$ is constructed from a finite number of function symbols, predicate symbols and axioms.
Thus for every theorem $\varphi \sststile{}{V} \psi$ of $\colim_{i \in I} T_i$ there exists a theory $T_i$ such that $\varphi \sststile{}{V} \psi$ is a theorem of $T_i$.
Note that the additional axioms of $\colim_{i \in I} T_i$ that was added for every $f : T_i \to T_j$ are always true in $T_j$.

Let $h : T \to \colim_{i \in I} T_i$ be a morphism from a $\lambda$-small theory $T$ to a $\lambda$-directed colimit of theories $\{ T_i \}_{i \in I}$.
Since $T$ is $\lambda$-small, there exists a theory $T_i$ such that for every function symbol $\sigma$, predicate symbol $R$ and axiom $\varphi \sststile{}{V} \psi$ of $T$,
restricted terms $h(\sigma(x_1, \ldots x_n))$ and formulae $h(R(x_1, \ldots x_n))$ belong to $T_i$, and $h(\varphi) \sststile{}{V} h(\psi)$ is a theorem of $T$.
Thus $h$ factors through $T_i$.

Let $h_1,h_2 : T \to T_i$ be morphisms such that $g_i \circ h_1 = g_i \circ h_2$, where $g_i : T_i \to \colim_{i \in I} T_i$.
Then for every function symbol $\sigma$ of $T$, sequent
\[ \sststile{}{x_1, \ldots x_n} h_1(\sigma(x_1, \ldots x_n)) \cong h_2(\sigma(x_1, \ldots x_n)) \]
is a theorem of $\colim_{i \in I} T_i$.
But we already know that there exists a theory $T_j$ such that $i \leq j$ and this sequent is a theorem of $T_j$.
The same is true for every predicate symbol of $T$.
It follows that $f \circ h_1 = f \circ h_2$, where $f : T_i \to T_j$.

Now, let us prove that $B/\PSt_{\mathcal{S}_0}$ is locally finitely presentable.
We only need to show that every theory in $B/\PSt_{\mathcal{S}_0}$ is a $\lambda$-directed colimit of its $\lambda$-small subtheories.
Let $T$ be a theory, and let $\{ f_i : T_i \to T' \}_{i \in I}$ be a cocone over the diagram of $\lambda$-small subtheories of $T$.
For every basic function or predicate symbol $p$ of $T$,
there is a finite subtheory $T_p$ of $T$ which contains symbols and axioms of $B$ and one additional symbol $p$ and no other axiom.
A morphism $h$ of cocones $T$ and $T'$ must commute with morphisms from $T_p$.
Thus it must be defined as $h(p(x_1, \ldots x_n)) = f_p(p(x_1, \ldots x_n))$; hence it is unique.
To prove that this defines a morphism, we need to show that $h$ preserves axioms of $T$.
But every axiom involves only a finite number of symbols of $T$.
Hence there exists a subtheory $T_i$ of $T$ which consists of these symbols and this axiom.
Since $f_i$ is a morphism of theories, this axiom also holds in $T'$.

Finally, let us prove that every $\lambda$-presentable theory $T$ in $B/\PSt_{\mathcal{S}_0}$ is isomorphic to a $\lambda$-small theory.
Consider the identity map $id_T : T \to T$.
Since $T$ is a $\lambda$-directed colimit of its $\lambda$-small subtheories, $id_T$ factors through some $\lambda$-small subtheory $T'$ of $T$.
Thus we have maps $f : T \to T'$ and $g : T' \to T$ such that $g \circ f = id_T$.
Since $T$ is a coequalizer of $f \circ g$ and $id_{T'}$, it is isomorphic to the coequalizer of $f \circ g$ and $id_{T'}$ as constructed above, which is a $\lambda$-small theory.
\end{proof}

\begin{cor}
The categories of stable and $c$-stable theories and categories of (stable, $c$-stable) algebraic dependent type theories are all locally finitely presentable.
\end{cor}
\begin{proof}
Each of this categories is a full reflective subcategory of the category of prestable theories closed under all colimits.
It follows from the previous proposition that they are locally finitely presentable.
\end{proof}

\subsection{Model structure}

Let $T_I = \coeT^{l'}_2 + \PathT + \wUA$ be the theory defined in \cite{alg-models}.
In this subsection we define a model structure on the category $T_I/\algtt$ of algebraic dependent type theories under $T_I$.

To construct this model structure, we need to recall a few definitions from \cite{f-model-structures}.
A reflexive cylinder object $C_U(V)$ for a map $i : U \to V$ is any factorization of $[id_V,id_V] : V \amalg_U V \to V$.
Maps $f,g : V \to X$ are homotopic relative to a cylinder object $[\cyli_0,\cyli_1] : V \amalg_U V \to C_U(V)$, if there exists a map $h : C_U(V) \to X$
such that $h \circ \cyli_0 = f$ and $h \circ \cyli_1 = g$.
In this case we will write $f \sim_i g$.
We say that a map $f : X \to Y$ has RLP up to $\sim_i$ with respect to $i : U \to V$ if for every commutative square of the form
\[ \xymatrix{ U \ar[r]^u \ar@{}[dr]|(.7){\sim_i} \ar[d]_i & X \ar[d]^f \\
              V \ar[r]_v \ar@{-->}[ur]^g                  & Y,
            } \]
there is a dotted arrow $g : V \to X$ such that $g \circ i = u$ and $(f \circ g) \sim_i v$.
We will say that a map has RLP up to relative homotopy with respect to a set $\I$ of maps if it has RLP up to $\sim_i$ with respect to every $i \in \I$.

We will also need the following theorem from \cite{f-model-structures}:
\begin{thm}[model-structure]
Let $\C$ be a complete and cocomplete category, and let $\I$ be a set of maps of $\C$
such that the domains and the codomains of maps in $\I$ are small relative to $\Icell$.
For every $i : U \to V \in \I$, choose a reflexive relative cylinder object $C_U(V)$
such that $[\cyli_0,\cyli_1] : V \amalg_U V \to C_U(V) \in \Icof$.
Let $\J_\I = \{\ \cyli_0 : V \to C_U(V)\ |\ i : U \to V \in \I \ \}$, and
let $\we_\I$ be the set of maps which have RLP up to relative homotopy with respect to $\I$.

Suppose that for all composable $f \in \Jcell[_\I] \cup \we_\I$ and $g$, if $g \circ f \in \we_\I$, then $g \in \we_\I$.
Then there exists a cofibrantly generated model structure on $\C$ with $\I$ as a set of generating cofibrations,
$\J_\I$ as a set of generating trivial cofibrations, and $\we_\I$ as a class of weak equivalences.
\end{thm}

For every sequence $(p_1,n_1), \ldots (p_{k+1},n_{k+1})$ of sorts, let $T_{(p_1,n_1), \ldots (p_{k+1},n_{k+1})}$ be the theory
with function symbols $\sigma_i : (p_1,n_1) \times \ldots \times (p_{i-1},n_{i-1}) \to (d_{p_i},n_i)$ for every $1 \leq i \leq k$,
$\sigma_{k+1} : (p_1,n_1) \times \ldots \times (p_k,n_k) \to (p_{k+1},n_{k+1})$,
and axioms $\varphi_1 \land \ldots \land \varphi_i \sststile{}{x_1, \ldots x_i} \sigma_{i+1}(x_1, \ldots x_i)\!\downarrow$ for every $1 \leq i \leq k$,
where $\varphi_j$ equals to $e_{p_j}(x_j) = \sigma_j(x_1, \ldots x_{j-1})$.
Let $\I$ be the set of maps of the form $T_{l, (d_p,n)} \to T_{l, (p,n)}$, where $l = s_1, \ldots s_k$ is any sequence of sorts,
$\sigma_i$ maps to $\sigma_i$ for every $1 \leq i \leq k$, and $\sigma_{k+1}$ maps to $e_p(\sigma_{k+1})$.
Let $\I_0 \subseteq \I$ be the subset which consists of the maps $T_{l, (d_p,n)} \to T_{l, (p,n)}$ such that $l$ is empty.

For every map in $\I$, we need to define a relative cylinder object for it.
Let $C_{T_{l,(\ty,n)}}(T_{l,(\tm,n)})$ be the theory with the same symbols and axioms as $T_l$,
three additional function symbol $\sigma, \sigma', h : s_1 \times \ldots \times s_k \to (\tm,n)$,
and axioms making $h$ into a relative homotopy between $\sigma$ and $\sigma'$ with respect to $\varphi_1 \land \ldots \land \varphi_k$.
Analogously, we define $C_{T_{l,(\ctx,n)}}(T_{l,(\ty,n)})$ to be the theory with the same symbols and axioms as $T_l$,
seven additional function symbols $\sigma,\sigma' : s_1 \times \ldots \times s_k \to (\ty,n)$, $f,g,g',p,q : s_1 \times \ldots \times s_k \to (\tm,n+1)$,
and axioms making $(f,g,g',p,q)$ into a relative homotopy between $\sigma$ and $\sigma'$ with respect to $\varphi_1 \land \ldots \land \varphi_k$.
Maps $\cyli_0,\cyli_1 : T_{l,(\tm,n)} \to C_{T_{l,(\ty,n)}}(T_{l,(\tm,n)})$ and their retraction
$s : C_{T_{l,(\ty,n)}}(T_{l,(\tm,n)}) \to T_{l,(\tm,n)}$ are defined in the obvious way.

\begin{remark}[triv-fib-lift]
By \rprop{morita-char}, a map has RLP up to relative homotopy with respect to $\I$ if and only if it is a Morita equivalence.
Similarly, \rprop{syn-equiv-char} implies that a map has RLP up to relative homotopy with respect to $\I_0$ if and only if it is a syntactic equivalence.
\end{remark}

\begin{lem}[jcell-morita]
Let $f : X \to Y$ be a pushout of $\cyli_0 : T_{l,(\tm,n)} \to C_{T_{l,(\ty,n)}}(T_{l,(\tm,n)})$ (in the category of $I$-stable theories under $T_I$)
and let $g : Y \to X$ be the retraction of $f$ which is the pushout of $s$.
Let $\varphi$ be a formula of $X$ such that for every predicate symbol $R$ occurring in $\varphi$,
sequent $R(x_1, \ldots x_k) \sststile{}{x_1, \ldots x_k} \alpha(L(R))(I, I \times x_1, \ldots I \times x_k)$ is derivable in $X$.

Then for every term $t$ of $Y$ such that $f(\varphi) \sststile{}{V} t\!\downarrow$, terms $t$ and $f(g(t))$ are relatively homotopic with respect to $f(\varphi)$.
\end{lem}
\begin{proof}
This lemma is analogous to \cite[Lemma~3.7]{alg-models}.
We defined there a function $h : \Term_Y(V)_{(p,n)} \to \Term_Y(L(V))_{(p,n+1)}$ such that $h$ preserves theorems in the sense that
if $\chi \sststile{}{V} \psi$ is a theorem of $Y$, then $h(\chi) \land \bigwedge_{x \in L(V)} \ctx^n(x) = I \sststile{}{L(V)} h(\psi)$ is also a theorem.
Note that $h(f(\varphi)) = \alpha(L(f(\varphi)))$ since $f(\varphi)$ contains only symbols of $X$.
The condition we put on $\varphi$ implies that sequent $f(\varphi) \sststile{}{V} h(f(\varphi))[\rho]$ is derivable, where $\rho(x) = I \times x$.
Thus we have the following theorem: $f(\varphi) \sststile{}{V} h(t)[\rho]\!\downarrow$.

Moreover, we have theorems $h(t)[\rho]\!\downarrow\ \sststile{}{V} h(t)[\rho][left] = f(g(t))$ and $h(t)[\rho]\!\downarrow\ \sststile{}{V} h(t)[\rho][right] = t$
(here, $[\rho]$ is an operation of substitution on terms and $[left]$ and $[right]$ are derived function symbols in the theory; we are sorry for this clash of the notation).
Thus $h(t)[\rho]$ gives us the required homotopy between $t$ and $f(g(t))$.
\end{proof}

\begin{thm}[theories-model-structure]
There exists a model structure on the category of $I$-stable algebraic dependent type theories under $T_I$
with $\I$ as the set of generating cofibrations, Morita equivalences as weak equivalences, and in which all objects are fibrant.
We call it \emph{the Morita model structure}.
\end{thm}
\begin{proof}
Note that the set $\we_\I$ consists of Morita equivalences.
Since Quillen equivalences satisfy the 2-out-of-3 property, by \rprop{morita-quillen}, Morita equivalences between theories under $T_I$ also satisfy it.
Since the codomains of maps in $\I$ are finite, Morita equivalences are closed under transfinite compositions.
Thus by \rthm{model-structure}, we just need to prove that pushouts of maps $\cyli_0 : T_{l,(p,n)} \to C_{T_{l,(d_p,n)}}(T_{l,(p,n)})$ are Morita equivalences.
Let $f : X \to Y$ be a pushout of $\cyli_0$ and let $g : Y \to X$ be its retract.
Let $\varphi$ be a formula of $X$ which does not contain any predicate symbols and let $A$ be a term of $X$ such that $\varphi \sststile{}{V} A\!\downarrow$.
Let $a$ be a term of $Y$ such that $f(\varphi) \sststile{}{V} e_p(a) = f(A)$.
If we define $a'$ as $g(a)$, then $\varphi \sststile{}{V} e_p(a') = A$ and the fact that $f(a')$ and $a$ are relatively homotopic follows from \rlem{jcell-morita}.
\end{proof}

\Rlem{jcell-morita} implies that trivial cofibrations satisfying a mild additional condition are strict Morita equivalences:

\begin{prop}
Let $f : T_1 \to T_2$ be a trivial cofibration such that for every predicate symbol $R$ of $T_1$,
sequent $R(x_1, \ldots x_k) \sststile{}{x_1, \ldots x_k} \alpha(L(R))(I, I \times x_1, \ldots I \times x_k)$ is derivable.
Then $f$ is a strict Morita equivalence.
\end{prop}
\begin{proof}
Since trivial cofibrations are retracts of maps in $\Jcell[_\I]$ and strict Morita equivalences are closed under retracts,
we just need to prove that maps in $\Jcell[_\I]$ are strict Morita equivalences.
Since strict Morita equivalences are closed under transfinite compositions, we just need to prove this for maps $f$ which are pushouts of maps in $\Jcell[_\I]$.
Moreover, since maps in $\Jcell[_\I]$ do not change the set of predicate symbols,
we may assume that the domain and the codomain of $f$ satisfy the same condition on the predicate symbols as $T_1$.
Now, \rlem{jcell-morita} implies that such maps are strict Morita equivalences.
\end{proof}

Note that the domains and the codomains of maps in $\I$ do not have any predicate symbols.
Thus cofibrant objects also do not have them (to be precise, they are isomorphic to theories without predicate symbols).
So it seems rather pointless to have predicate symbols at this point.
We can consider the full subcategory $\algtt_f$ of $\algtt$ on theories without predicate symbols
(and without function symbols of the form $\sigma : s_1 \times \ldots \times s_k \to (\ctx,0)$).
\Rprop{theories-presentable} still holds for $\algtt_f$, so this category is locally finitely presentable.
There is a model structure on $T_I/\algtt_f$ in which the classes of cofibrations, fibrations, and weak equivalences
are the intersections of the corresponding classes in $T_I/\algtt$ with the class of morphisms of $\algtt_f$.
This model category has the same sets of generating cofibrations and generating trivial cofibrations as $\algtt$.

\begin{prop}
The inclusion functor $T_I/\algtt_f \to T_I/\algtt$ has a right adjoint and this adjunction is a Quillen equivalence.
\end{prop}
\begin{proof}
We will say that a theory $T$ \emph{has enough function symbols} if for every restricted term $t$ of sort $s$ with free variables $x_1 : s_1$, \ldots $x_k : s_k$,
there is a function symbol $\sigma : s_1 \times \ldots \times s_k \to s$ such that sequent $\sststile{}{x_1, \ldots x_k} t \cong \sigma(x_1, \ldots x_k)$ is derivable.
Note that every theory $T$ is isomorphic to a theory $T'$ with enough function symbols.
Indeed, function symbols of $T'$ are just terms of the original theory and
axioms of $T'$ are axioms of $T$ together with axioms that say that the new terms are equivalent to the old ones.

Thus we may restrict and corestrict the inclusion functor $i : \algtt_f \to \algtt$ to the full subcategories of $\algtt_f$ and $\algtt$ on theories with enough function symbols.
We will denote this functor by $i' : \algtt'_f \to \algtt'$.
Now, it is easy to describe a right adjoint to $i'$.
For every theory $T \in \algtt'$, let $r'(T)$ be the theory with the same function symbols as $T$, no predicate symbols,
and with the set of axioms which consists of all theorems of $T$ which do not involve predicate symbols.
Then $r'$ is a functor $\algtt' \to \algtt'_f$.
It is easy to see that $r'$ is right adjoint to $i'$.
Since $i'(r'(T))$ and $T$ have the same sets of terms and theorems (which do not involve predicate symbols), the counit $\epsilon_T : i'(r'(T)) \to T$ is a trivial fibration.

Finally, note that the inclusion functor $T_I/i : T_I/\algtt_f \to T_I/\algtt$ preserves and reflects cofibrations and weak equivalences.
Moreover, it has a right adjoint and the counit of the adjunction is a trivial fibrations.
Thus this adjunction is a Quillen equivalence.
\end{proof}

Now, let us return to the original problem of the absence of predicate symbols in cofibrant objects.
Instead of forbidding predicate symbols completely, we can enlarge the class of cofibrations to include predicate symbols.
For every sequence of sorts $s_1, \ldots s_k$, let $P^1_{s_1, \ldots s_k}$ be the theory under $T_I$ with one additional predicate symbol $P : s_1 \times \ldots \times s_k$.
Also, we define the following theories:
\begin{align*}
P^2_{s_1, \ldots s_k} & = P^1_{s_1, \ldots s_k} \cup \{ Q : s_1 \times \ldots \times s_k, P(x_1, \ldots x_k) \sststile{}{x_1, \ldots x_k} Q(x_1, \ldots x_k) \} \\
P^3_{s_1, \ldots s_k} & = P^2_{s_1, \ldots s_k} \cup \{ R : s_1 \times \ldots \times s_k, Q(x_1, \ldots x_k) \sststile{}{x_1, \ldots x_k} R(x_1, \ldots x_k) \}
\end{align*}
Let $\I^P$ be the union of $\I$ and maps of the form $P^1_l \to P^2_l$, $P \mapsto Q$ and $P^2_l \to P^3_l$, $P \mapsto P$, $Q \mapsto R$, where $l$ is any sequence of sorts.
We define a relative cylinder object for the map $P^\alpha_l \to P^{\alpha+1}_l$ as $P^{\alpha+1}_l \amalg_{P^\alpha_l} P^{\alpha+1}_l$, where $\alpha \in \{ 1, 2\}$.
Thus any two maps $P^{\alpha+1}_l \to X$ are homotopic.
This implies that $\we_{\I^P} = \we_\I$.
To prove that there is a model structure on $T_I/\algtt$ with $\I^P$ as a set of generating cofibrations,
we just need to show that pushouts of maps $P^{\alpha+1}_l \to P^{\alpha+1}_l \amalg_{P^\alpha_l} P^{\alpha+1}_l$ are Morita equivalences.
But this is obvious since the domain and the codomain of such a pushout have the same sets of terms and axioms.
Of course, the identity functor determine a Quillen equivalence between the two model structures on $T_I/\algtt$.

Finally, let us discuss another model structure on the category of $I$-stable theories under $T_I$, which we call \emph{the syntactic model structure}.
The weak equivalences of this model structure are syntactic equivalences, $\I_0$ is a set of generating cofibrations, and every object is fibrant in this model structure.

Recall that for every theory $T$, we have a left adjoint functor $\Lang_T : \Mod{T} \to T/\algtt$ with a right adjoint $\Syn_T : T/\algtt \to \Mod{T}$.
Note that a map $f : T_1 \to T_2$ of theories under $T_I$ is a syntactic equivalence if and only if $\Syn_{T_I}(f)$ is a weak equivalence of models.
Thus we can transfer the model structure on $\Mod{T_I}$ to a model structure on $T_I/\algtt$.
To do this, we need to prove that $\Lang_{T_I}$ maps trivial cofibrations to syntactic equivalences.
But we already proved that it actually maps them to Morita equivalences.
Note that the identity functor on $T_I/\algtt$ is a left Quillen functor from the syntactic model structure to the Morita model structure
since it preserves generating cofibrations and generating trivial cofibrations.

Let $T$ be an $I$-stable theory under $T_I$.
Then the adjunction $\Lang_T \dashv \Syn_T$ is a Quillen equivalence between the model structure on $\Mod{T}$ and the syntactic model structure on $T/\algtt$.
To prove this, we just need to show that the unit of the adjunction is a weak equivalence.
But it is actually an isomorphism since $\Lang_T$ is full and faithful.

\section{Characterization of lifting properties}
\label{sec:triv-fib}

In this section we prove several useful lemmas that characterize trivial fibrations in various model structures that we considered in previous section.
Since all of these notations are defined in terms of lifting properties, we will prove general results about them.
Indeed, if we think about a map that has the weak lifting property with respect to a pair $V,\varphi$ as a weak equivalence,
then a map having the lifting property with respect to this pair can be thought of as a trivial fibration.

If we want to show that a map $f : T_1 \to T_2$ has the lifting property with respect to some pair $V,\varphi$, we could try to prove this by induction on terms.
If we do this, then we need to prove the lifting property only for function symbols (and not arbitrary terms).
Also, we need to know that we can transport terms in $T_1$ along equalities in $T_2$.
That is, if $a$ is a term of type $A$ in $T_1$ and $B$ is a type such that $f(A) = f(B)$, then there should be a term $b$ of type $B$ such that $f(a) = f(b)$ (and the same property should hold for types).
This is proved in the following lemma:

\begin{lem}[eq-char-fib]
Let $V$ be a set of variables and let $\varphi$ be a formula of a theory $T_1$ with free variables in $V$.
Let $f : T_1 \to T_2$ be a morphism of theories such that the following conditions hold:
\begin{enumerate}
\item \label{it:char-fun-fib} For every function symbol $\sigma$ of $T_2$, there exist terms $A_1$, \ldots $A_k$, $t$ of $T_1$ such that $\FV(A_i) \subseteq \{ x_1, \ldots x_{i-1} \}$ and the following sequents are derivable:
\begin{align*}
\bigwedge_{1 \leq i < j} e_{p_i}(x_i) = A_i & \sststile{T_1}{x_1, \ldots x_{j-1}} A_j\!\downarrow \text{ for every } 1 \leq j \leq k \\
\bigwedge_{1 \leq i \leq k} e_{p_i}(x_i) = A_i & \sststile{T_1}{x_1, \ldots x_k} t\!\downarrow \\
\sigma(x_1, \ldots x_k)\!\downarrow & \sststile{T_2}{x_1, \ldots x_k} f(t) = \sigma(x_1, \ldots x_k) \land \bigwedge_{1 \leq i \leq k} e_{p_i}(x_i) = f(A_i)
\end{align*}
\item \label{it:char-type-fib} For all terms $A$, $B$, and $a$ of $T_1$ such that $\varphi \sststile{T_1}{V} e_p(a) = A$, $\varphi \sststile{T_1}{V} B\!\downarrow$, and $f(\varphi) \sststile{T_2}{V} f(A) = f(B)$,
there exists a term $b$ such that $\varphi \sststile{T_1}{V} e_p(b) = B$ and $f(\varphi) \sststile{T_2}{V} f(a) = f(b)$.
\end{enumerate}
Then $f$ has the lifting property with respect to $V,\varphi$.
\end{lem}
\begin{proof}
Let $A$ and $a$ be terms such that $\varphi \sststile{T_1}{V} A\!\downarrow$ and $f(\varphi) \sststile{T_2}{V} e_p(a) = f(A)$.
Then we construct the required lifting by induction on $a$.
If $a = x$ is a variable, then \eqref{it:char-type-fib} implies that that the required lifting exists.

Now, suppose that $a = \sigma(a_1, \ldots a_k)$.
Let $A_1$, \ldots $A_k$, $t$ be terms as described in \eqref{it:char-fun-fib}.
By the induction hypothesis, there exist terms $a_1'$, \ldots $a_k'$ such that $\varphi \sststile{T_1}{V} e_{p_i}(a_i') = A_i[x_1 \repl a_1', \ldots x_{i-1} \repl a_{i-1}']$ and $f(\varphi) \sststile{T_2}{V} f(a_i') = a_i$.
Let $a' = t[x_1 \repl a_1', \ldots x_k \repl a_k']$.
Since $\varphi \sststile{T_1}{V} a'\!\downarrow$ and $f(\varphi) \sststile{T_2}{V} f(a') = f(a)$, \eqref{it:char-type-fib} implies that there exists the required lifting.
\end{proof}

\begin{remark}[stable-eq-char-fib]
If $T_2$ is a stable contextual theory, then it is enough to check the first condition of the previous lemma for basic function symbols.
It is easy to see that this implies the general case.
\end{remark}

We can simplify the second condition of \rlem{eq-char-fib}.
It is enough to show that given two types $A$ and $B$ in $T_1$ such that $f(A) = f(B)$, the identity map on $f(A)$ lifts to a term in $T_1$:

\begin{lem}[char-type-fib]
Let $V$ be a set of variables and let $\varphi$ be a formula of a theory $T_1$ with free variables in $V$.
Let $f : T_1 \to T_2$ be a morphism of theories such that the following condition holds.
For every pair of terms $A$ and $B$ of $T_1$ of sort $(\ty,n)$ such that
$\varphi \sststile{T_1}{V} \ft(A) = \ft(B)$ and $f(\varphi) \sststile{T_2}{V} f(A) = f(B)$,
there exists a term $b$ such that $\varphi \sststile{T_1}{V} A \vdash b : B$ and $f(b)$ equals to $v_0(f(A))$.
Then the second condition of \rlem{eq-char-fib} holds.
\end{lem}
\begin{proof}
First, note that if $\Gamma$ and $\Delta$ are context such that $\varphi \sststile{T_1}{V} \Gamma\!\!\downarrow\!\land\:\Delta\!\!\downarrow$ and $f(\varphi) \sststile{T_2}{V} f(\Gamma) = f(\Delta)$,
then there is a morphism of context $d : \Gamma \to \Delta$ such that $f(d)$ is the identity morphism.
We prove this by induction on the length of contexts.
There is a unique morphism between empty contexts.
If contexts are not empty, then we have a morphism $d : \ft(\Gamma) \to \ft(\Delta)$ by the induction hypothesis.
Since $f(d^*(\Delta))$ equals to $f(\Gamma)$, we have a term $\Gamma \vdash b : d^*(\Delta)$ such that $f(b)$ equals to $v_0(f(\Gamma))$ by assumption.
Thus, $d,b$ is the required morphism of contexts $\Gamma$ and $\Delta$.

For all terms $A$, $B$, and $a$ of $T_1$ such that $\varphi \sststile{T_1}{V} e_p(a) = A$ and $f(\varphi) \sststile{T_2}{V} f(A) = f(B)$,
there exists a term $b$ such that $\varphi \sststile{T_1}{V} e_p(b) = B$ and $f(\varphi) \sststile{T_2}{V} f(a) = f(b)$.
We have a morphism $d : A \to B$ such that $f(d)$ equals to the identity morphism.
if $p = \ty$, then we can define $b$ as $d^*(a)$.
if $p = \tm$, then let $B = d^*(\ty(a))$.
Since $f(B) = f(\ty(a))$, we have a term $\ty(a) \vdash b' : B$ such that $f(b')$ equals to $v_0(f(B))$ by assumption.
Then we can define $b$ as $b'[d^*(a)]$.
\end{proof}

If we want to simplify the first condition of \rlem{eq-char-fib}, then we need certain assumptions on the theory $T_2$.
We will say that a theory \emph{has well-defined function symbols} if (it is isomorphic to a theory such that) there exists a well-founded relation on the set of function symbols such that,
for every function symbol $\sigma$, either $\sigma$ equals to one of the function symbols $\ty_n$, $\ft_n$ or there exist terms $A_1$, \ldots $A_k$ satisfying the following conditions:
\begin{enumerate}
\item All function symbols that occur in $A_1$, \ldots $A_k$ are less than $\sigma$.
\item $\FV(A_i) \subseteq \{ x_1, \ldots x_{i-1} \}$ for every $1 \leq i \leq k$.
\item The following sequents are derivable:
\begin{align*}
\bigwedge_{1 \leq i < j} e_{p_i}(x_i) = A_i & \sststile{}{x_1, \ldots x_{j-1}} A_j\!\downarrow \text{ for every } 1 \leq j \leq k \\
\bigwedge_{1 \leq i \leq k} e_{p_i}(x_i) = A_i & \ssststile{}{x_1, \ldots x_k} \sigma(x_1, \ldots x_k)\!\downarrow
\end{align*}
\end{enumerate}
We will say that terms $A_1$, \ldots $A_k$ \emph{define} the function symbol $\sigma$.

This condition is easy to check and most of the theories that occur in practice satisfy it.
An example of a theory that does not satisfy it appeared in \cite{alg-models}: it is the theory of filler operations.
The rest of the theories that appear in \cite{alg-models} and also all of the theories in this paper and in \cite{alg-tt} have well-defined function symbols.

Now, we can simplify the first condition of \rlem{eq-char-fib}.
This condition requires to construct not only a lift of a function symbol but also terms in the left hand side of the sequent defining this term.
The following lemma shows that it is enough to construct this lift with an arbitrary formula as the left hand side of this sequent:

\begin{lem}[eq-comp-char-fib]
Let $V$ be a set of variables and let $\varphi$ be a formula of a theory $T_1$ with free variables in $V$.
Let $f : T_1 \to T_2$ be a morphism of theories such that $T_2$ has well-defined function symbols.
Suppose that the second condition of \rlem{eq-char-fib} holds.
Moreover, suppose that, for every function symbol $\sigma$ of $T_2$, either the first condition of \rlem{eq-char-fib} holds for $\sigma$ or,
for every pair $\psi,\{ x_1, \ldots x_k \}$ in $P_M$, if $f(\psi) \sststile{T_2}{x_1, \ldots x_k} \sigma(x_1, \ldots x_k)\!\downarrow$ is derivable, then there exists a term $t$ of $T_1$ such that the following sequents are derivable:
\begin{align*}
\psi & \sststile{T_1}{x_1, \ldots x_k} t\!\downarrow \\
f(\psi) & \sststile{T_2}{x_1, \ldots x_k} \sigma(x_1, \ldots x_k) = f(t)
\end{align*}
Then the first condition of \rlem{eq-char-fib} holds.
\end{lem}
\begin{proof}
First, note that if the first condition of \rlem{eq-char-fib} holds for some subset of function symbols of $T_2$, then we still can lift terms constructed from function symbols from this subset.
The proof of this fact is the same as the proof of \rlem{eq-char-fib}.
We also note that symbols $\ty_n$ and $\ft_n$ satisfy the first condition of \rlem{eq-char-fib}, so we may assume that they are less than every other symbol.

Now, we can prove by well-founded induction on $\sigma$ that the first condition of \rlem{eq-char-fib} holds.
If it holds for $\sigma$, then we are done.
Otherwise, let $A_1$, \ldots $A_k$ be terms that define $\sigma$.
By the induction hypothesis, there exist lifts $A_1'$, \ldots $A_k'$ of these terms (we first lift $\ft^{n_i}(e_{p_i}(A_i))$, then $\ft^{n_i-1}(e_{p_i}(A_i))$, and so on; finally, we can lift $A_i$).
If we let $\psi = (\bigwedge_{1 \leq i \leq k} e_{p_i}(x_i) = A_i')$, then $f(\psi) \sststile{T_2}{x_1, \ldots x_k} \sigma(x_1, \ldots x_k)\!\downarrow$.
Thus, by assumption, we have a term $t$ such that $\psi \sststile{T_1}{x_1, \ldots x_k} t\!\downarrow$ and $f(\psi) \sststile{T_2}{x_1, \ldots x_k} f(t) = \sigma(x_1, \ldots x_k)$.
Since $\sigma(x_1, \ldots x_k)\!\downarrow\ \sststile{T_2}{x_1, \ldots x_k} f(\psi)$, we are done.
\end{proof}

Finally, we can show that the conditions of the previous lemmas are often not only sufficient, but also necessary:

\begin{prop}[eq-comp-char-fib]
Let $P$ be a set such that $P_M \subseteq P \subseteq P_S$.
If $f : T_1 \to T_2$ is a morphism of theories such that $T_2$ has well-defined function symbols, then the following conditions are equivalent:
\begin{enumerate}
\item \label{it:lem-fib-sec} The first condition of \rlem{eq-char-fib} and conditions of \rlem{char-type-fib} hold for all pairs in $P$.
\item \label{it:lem-fib} Both conditions of \rlem{eq-char-fib} hold for all pairs in $P$.
\item \label{it:eq-fib} $f$ has the lifting property with respect to $P$.
\item \label{it:cond-fib} Conditions of \rlem{eq-comp-char-fib} hold for all pairs in $P$.
\end{enumerate}
\end{prop}
\begin{proof}
The implication \eqref{it:lem-fib-sec} $\implies$ \eqref{it:lem-fib} follows from \rlem{char-type-fib}.
The implication \eqref{it:lem-fib} $\implies$ \eqref{it:eq-fib} follows from \rlem{eq-char-fib}.
The implication \eqref{it:eq-fib} $\implies$ \eqref{it:cond-fib} is obvious since conditions in \eqref{it:cond-fib} are just special cases of the lifting property.
Finally, \eqref{it:cond-fib} implies \eqref{it:lem-fib} by \rlem{eq-comp-char-fib}, and conditions of \rlem{char-type-fib} are a special case of the lifting property, hence they follow from \eqref{it:eq-fib}.
\end{proof}

\section{Confluent theories}
\label{sec:confluent}

The axioms of type theories that occur in practice often can be divided in two parts: the first part determines when function symbols are defined and the second part is defined in terms of some reduction relation,
which often satisfies some additional properties such as confluence.
In this section we define confluent theories as theories in this form.
We also prove that Morita equivalences between them are easier to construct.

\subsection{Theories with separated axioms}

We will say that a theory has separated axioms if the set of axioms of this theory consists of three disjoint subsets $\mathcal{A}_d$, $\mathcal{A}'_d$, and $\mathcal{A}_e$ such that the following conditions hold:
\begin{enumerate}
\item The set $\mathcal{A}_d$ consists of axioms of the form $\varphi \sststile{}{x_1, \ldots x_k} \sigma(x_1, \ldots x_k)\!\downarrow$ and, for every $\sigma$, there is exactly one axiom of this form.
We will denote the left hand side of this axiom by $\varphi_\sigma$.
\item The set $\mathcal{A}'_d$ consists of (not necessary all) sequents of the form $\sigma(x_1, \ldots x_k)\!\downarrow\ \sststile{}{x_1, \ldots x_k} \varphi_\sigma$.
\item \label{it:sep-ax} For every axiom $\varphi \sststile{}{V} \psi$ in $\mathcal{A}_e$ and every subterm $\sigma(t_1, \ldots t_k)$ of $\psi$,
the sequent $\varphi \sststile{}{V} \varphi_\sigma[t_1/x_1, \ldots t_k/x_k]$ is derivable from the axioms $\mathcal{A}_d \cup \mathcal{A}_e$.
\end{enumerate}
We will say that a theory with separated axioms is \emph{minimal} (resp., \emph{maximal}) if $\mathcal{A}'_d$ is empty (resp., consists of all sequents of the specified form).

\begin{lem}[der-separated-closed]
Let $T$ be a theory with separated axioms.
If a sequent $\sststile{}{} \psi$ is derivable in $T$, then it is derivable from $\mathcal{A}_d \cup \mathcal{A}_e$.
\end{lem}
\begin{proof}
First, let us prove the following fact.
If a sequent $\sststile{}{} \psi$ is derivable from $\mathcal{A}_d \cup \mathcal{A}_e$ and a term $\sigma(t_1, \ldots t_k)$ is a subterm of $\psi$,
then the sequent $\sststile{}{} \varphi_\sigma[t_1/x_1, \ldots t_k/x_k]$ is derivable from $\mathcal{A}_d \cup \mathcal{A}_e$.
We prove this by induction on the derivation of $\sststile{}{} \psi$ in the natural deduction system.
Let us consider the case \axref{nl}:
\begin{center}
\AxiomC{$\sststile{}{} a = b$}
\AxiomC{$\sststile{}{} \psi[a/x]$}
\RightLabel{\axref{nl}}
\BinaryInfC{$\sststile{}{} \psi[b/x]$}
\DisplayProof
\end{center}
If $\sigma(t_1, \ldots t_k)$ is a subterm of $b$, then the required property follows from the induction hypothesis for $\sststile{}{} a = b$.
Otherwise, $\sigma$ belongs to $\psi$ and there exist terms $t_1'$, \ldots $t_k'$ and a formula $\psi'$ such that $t_i = t_i'[b/x]$, $\psi = \psi'[\sigma(t_1', \ldots t_k')/x]$, and $x \notin \FV(\psi')$.
The induction hypothesis implies that the sequent \[ \sststile{}{} \varphi_\sigma[t_1'[a/x]/x_1, \ldots t_k'[a/x]/x_k] \] is derivable.
Since the sequent $\sststile{}{} a = b$ is derivable, this implies that the required sequent $\sststile{}{} \varphi_\sigma[t_1/x_1, \ldots t_k/x_k]$ is also derivable.

Let us consider the inference rule for axioms from $\mathcal{A}_d$:
\smallskip
\begin{center}
\AxiomC{$\sststile{}{} s_i\!\downarrow$, $1 \leq i \leq n$}
\AxiomC{$\sststile{}{} \varphi_\tau[s_1/y_1, \ldots s_n/y_n]$}
\RightLabel{\axref{na}}
\BinaryInfC{$\sststile{}{} \tau(s_1, \ldots s_n)\!\downarrow$}
\DisplayProof
\end{center}
If $\sigma(t_1, \ldots t_k)$ is a subterm of $s_i$ for some $i$, then the required property follows from the induction hypothesis for $\sststile{}{} s_i\!\downarrow$.
Otherwise, $\tau(s_1, \ldots s_n) = \sigma(t_1, \ldots t_k)$ and the required property is obvious.
This inference rule for axioms from $\mathcal{A}_e$ follows from \eqref{it:sep-ax}.
The rest of the inference rules are trivial.

Now, we can prove the lemma.
We proceed by induction on the inference of $\sststile{}{} \psi$.
Most of the cases follow immediately from the induction hypothesis.
The only nontrivial case is the inference rule for axioms from $\mathcal{A}'_d$:
\smallskip
\begin{center}
\AxiomC{$\sststile{}{} t_i\!\downarrow$, $1 \leq i \leq k$}
\AxiomC{$\sststile{}{} \sigma(t_1, \ldots t_k)\!\downarrow$}
\RightLabel{\axlabel{na}}
\BinaryInfC{$\sststile{}{} \varphi_\sigma[t_1/x_1, \ldots t_k/x_k]$}
\DisplayProof
\end{center}
By the induction hypothesis, the sequent $\sststile{}{} \sigma(t_1, \ldots t_k)\!\downarrow$ is derivable from $\mathcal{A}_d \cup \mathcal{A}_e$
and the fact that we just proved implies that $\sststile{}{} \varphi_\sigma[t_1/x_1, \ldots t_k/x_k]$ is also derivable from these axioms.
\end{proof}

\begin{lem}[morita-separated]
Let $T$ be a theory and let $(\varphi_1 \land \ldots \land \varphi_n, \{ x_1, \ldots x_n \})$ be a pair in $P_M$.
If $T$ has separated axioms, then so does the following theory:
\[ T_n = T \cup \{ \sststile{}{} c_i\!\downarrow |\ 1 \leq i \leq n \} \cup \{ \sststile{}{} \varphi_i[c_1/x_1, \ldots c_i/x_i]\ |\ 1 \leq i \leq n \} \]
\end{lem}
\begin{proof}
Let $\varphi_i$ be equal to $e_{p_i}(x_i) = A_i$.
We proceed by induction on $n$.
Since $T_0 = T$, the case $n = 0$ holds by assumption.
Assume that $T_{n-1}$ has separated axioms.
To prove that $T_n$ also has separated axioms, we need to show that, for every subterm $\sigma(t_1, \ldots t_k)$ of $A_n[c_1/x_1, \ldots c_{n-1}/x_{n-1}]$,
the sequent $\sststile{}{} \varphi_\sigma[t_1/x_1, \ldots t_k/x_k]\!\downarrow$ is derivable from $\mathcal{A}_d \cup \mathcal{A}_e$.
Since the sequent $\sststile{}{} A_n[c_1/x_1, \ldots c_{n-1}/x_{n-1}]\!\downarrow$ is derivable in $T_{n-1}$, the sequents $\sststile{}{} \sigma(t_1, \ldots t_k)\!\downarrow$ and $\sststile{}{} \varphi_\sigma[t_1/x_1, \ldots t_k/x_k]$ are also derivable in it.
Since $T_{n-1}$ has separated axioms, \rlem{der-separated-closed} implies that $\sststile{}{} \varphi_\sigma[t_1/x_1, \ldots t_k/x_k]$ is derivable from $\mathcal{A}_d \cup \mathcal{A}_e$ in $T_{n-1}$ and hence in $T_n$.
\end{proof}

\begin{prop}[der-separated]
Let $T$ be a theory with separated axioms and let $(\varphi,V)$ be a pair in $P_M$.
If a sequent $\varphi \sststile{}{V} \psi$ is derivable in $T$, then it is derivable from $\mathcal{A}_d \cup \mathcal{A}_e$.
\end{prop}
\begin{proof}
\Rlem{mcf} implies that $\sststile{}{} \psi[c_1/x_1, \ldots c_n/x_n]$ is derivable in $T \cup \mathcal{A}'$, where $\mathcal{A}' = \{ \sststile{}{} c_i\!\downarrow |\ 1 \leq i \leq n \} \cup \{ \varphi_i[c_1/x_1, \ldots c_i/x_i]\ |\ 1 \leq i \leq n \}$.
By \rlem{morita-separated} and \rlem{der-separated-closed}, this sequent is derivable from $\mathcal{A}_d \cup \mathcal{A}_e \cup \mathcal{A}'$.
\Rlem{mcf} implies that $\varphi \sststile{}{V} \psi$ is derivable from $\mathcal{A}_d \cup \mathcal{A}_e$.
\end{proof}

The previous proposition implies that the maps between theories with separated axioms corresponding to inclusions of subtheories are Morita equivalences.
This shows that if we are interested in a theory with separated axioms, then we can work with either minimal or maximal theory corresponding to it instead.
In general, we prefer to work with the latter, but sometimes it is convenient to switch to the former.

Note that relative $\I$-cell complexes are minimal theories with separated axioms.
We can define another set $\I'$ of generating cofibrations such that relative $\I'$-cell complexes are maximal theories with separated axioms.
Recall that $T_{(p_1,n_1), \ldots (p_{k+1},n_{k+1})}$ is the theory with function symbols $\sigma_i : (p_1,n_1) \times \ldots \times (p_{i-1},n_{i-1}) \to (d_{p_i},n_i)$ for every $1 \leq i \leq k$,
$\sigma_{k+1} : (p_1,n_1) \times \ldots \times (p_k,n_k) \to (p_{k+1},n_{k+1})$, and axioms $\varphi_1 \land \ldots \land \varphi_i \sststile{}{x_1, \ldots x_i} \sigma_{i+1}(x_1, \ldots x_i)\!\downarrow$ for every $1 \leq i \leq k$,
where $\varphi_j$ equals to $e_{p_j}(x_j) = \sigma_j(x_1, \ldots x_{j-1})$.
The set $\I$ consists of maps of the form $T_{l, (d_p,n)} \to T_{l, (p,n)}$.
We define $T'_{s_1, \ldots s_{k+1}}$ as $T_{s_1, \ldots s_{k+1}}$ together with the axiom $\sigma_{k+1}(x_1, \ldots x_k)\!\downarrow\ \sststile{}{x_1, \ldots x_k} \varphi_1 \land \ldots \land \varphi_k$.
Let $\I'$ be the set of maps of the form $T_{l, (d_p,n)} \to T'_{l, (p,n)}$.

Note that $\I'$ is a retract of $\I$.
Indeed, we can define a map $T'_{s_1, \ldots s_{k+1}} \to T_{s_1, \ldots s_{k+1}}$ as $\sigma_{k+1}(x_1, \ldots x_k) \mapsto \sigma_{k+1}(x_1, \ldots x_k)|_{\varphi_1 \land \ldots \land \varphi_k}$
and a map $T_{s_1, \ldots s_{k+1}} \to T'_{s_1, \ldots s_{k+1}}$ as the obvious inclusion.
Then the composite $T'_{s_1, \ldots s_{k+1}} \to T_{s_1, \ldots s_{k+1}} \to T'_{s_1, \ldots s_{k+1}}$ is the identity morphism.
Thus, the class of $\I'$-cofibrations is a subclass of $\I$-cofibrations.
Moreover, the composite $T_{s_1, \ldots s_{k+1}} \to T'_{s_1, \ldots s_{k+1}} \to T_{s_1, \ldots s_{k+1}}$ is homotopic to the identity morphism.
Hence a map has the weak lifting property with respect to $\I$ if and only if it has this property with respect to $\I'$.

It follows that there is another model structure on the category $T_I/\algtt$ in which all objects are fibrant, weak equivalences are Morita equivalences, and cofibrations are $\I'$-cofibrations.
The identity functor is a Quillen equivalence between this model structure and the model structure that we constructed in the previous section.

\subsection{Confluent theories}

In this subsection we define confluent theories and prove their properties.
First, we need to define a few notions from the theory of abstract reduction systems.
For a general introduction to this topic we refer the reader to \cite{Terese,klop-trs,ohlebusch-advanced}.
\begin{enumerate}
\item An \emph{abstract reduction system} is a set $A$ together with a binary relation $\Rightarrow$ on it.
We will denote by $\Rightarrow^*$ the reflexive transitive closure of $\Rightarrow$.
If $\Rightarrow_1$ and $\Rightarrow_2$ are some relations, then we will write $\Rightarrow_1 \Rightarrow_2$ for the following relation:
$t \Rightarrow_1 \Rightarrow_2 t'$ if and only if there is a term $s$ such that $t \Rightarrow_1 s$ and $s \Rightarrow_2 t'$.
\item An element $a$ \emph{reduces} to an element $a'$ if $a \Rightarrow^* a'$.
A \emph{reduction sequence} is a finite or infinite sequence of elements $a_i$ such that $a_0 \Rightarrow a_1 \Rightarrow a_2 \Rightarrow \ldots$.
\item Two elements $a$ and $b$ are \emph{$\Rightarrow$-equivalent} if there is a sequence of elements $a_1$, \ldots $a_n$ such that $a = a_1$, $b = a_n$,
and, for every $1 \leq i < n$, either $a_i \Rightarrow a_{i+1}$ or $a_{i+1} \Rightarrow a_i$.
\item Two elements $a$ and $b$ are \emph{joinable} if there is an element $c$ such that $a \Rightarrow^* c$ and $b \Rightarrow^* c$.
We will also say that $a$ and $b$ are joinable under $\Rightarrow$ if the reduction relation is not clear from the context.
An element $a$ is \emph{confluent} if whenever $a \Rightarrow^* b$ and $a \Rightarrow^* c$ the terms $b$ and $c$ are joinable.
The system is confluent if every element is confluent.
Equivalently, the system is confluent if every pair of $\Rightarrow$-equivalent elements is joinable.
\item An element $a$ is a \emph{normal form} if there is no element $a'$ such that $a \Rightarrow a'$.
We will write $a \Rightarrow^\nf b$ if $a \Rightarrow^* b$ and $b$ is a normal form.
We will say that an element $a$ \emph{has a normal form} (or that it is \emph{weakly normalizable}) if $a \Rightarrow^\nf b$ for some $b$.
The system is \emph{weakly normalizing} if every element has a normal form.
\item An element $a$ is \emph{strongly normalizable} if there is no infinite reduction sequence sequence starting with $a$.
The system is \emph{strongly normalizing} if all elements are.
\item A subset $A'$ of $A$ is \emph{closed} under $\Rightarrow$ if $a' \in A'$ and $a' \Rightarrow a$ implies that $a \in A'$.
\end{enumerate}

A \emph{term rewriting system} is a binary relation $R$ on the set of terms of some theory such that the following conditions hold:
\begin{enumerate}
\item If $R(t,s)$, then $\FV(s) \subseteq \FV(t)$.
\item If $R(t,s)$, then $t$ is not a variable.
\end{enumerate}
A term rewriting system $R$ is \emph{left-linear} if, for every $t$ and $s$ such that $R(t,s)$, every variable occurs in $t$ at most once.

If $R$ is a term rewriting system, then we define the relation $\Rightarrow_R$ on the set of terms as follows:
if $R(t,s)$, then $c[x \repl t[x_1 \repl t_1, \ldots x_k \repl t_k]] \Rightarrow_R c[x \repl s[x_1 \repl t_1, \ldots x_k \repl t_k]]$ for all $c$, $x_1$, \ldots $x_k$, and $t_1$, \ldots $t_k$.
Every term rewriting system has the underlying abstract reduction system $(\Term_T,\Rightarrow_R)$.

We will say that a term $t$ of a theory $T$ is \emph{defined} with respect to a pair $(\varphi,V) \in P_M$ if the sequent $\varphi \sststile{}{V} t\!\downarrow$ is derivable.
Let $\Term_{T,\varphi}^d$ be the set of defined terms with respect to the pair $(\varphi,\FV(\varphi))$.
There is an abstract reduction system $\Rightarrow_\varphi$ on the set $\Term_{T,\varphi}^d$ defined as follows.
If $\varphi = \varphi_1 \land \ldots \land \varphi_n$ and $\varphi_i$ equals to $e_{p_i}(x_i) = t_i$, then $\Rightarrow_\varphi$ consists of pairs $(c[e_{p_i}(x_i)/y],c[t_i/y])$
for every $1 \leq i \leq n$ and every term $c$ such that there is exactly one occurrence of $y$ in $c$.

Axioms of a type theory are often presented in the form of a term rewriting system.
So there is a natural choice of an abstract reduction system on the set of terms of a type theory.
We axiomatize this situation in the following definition:

\begin{defn}[directed]
Let $T$ be a theory with separated axioms.
A \emph{reduction system} on $T$ is a choice of an abstract reduction system $\Rightarrow_{T,\varphi}$ on $\Term^d_{T,\varphi}$ for every pair $(\varphi,V) \in P_M$ such that the following conditions hold:
\begin{enumerate}
\item \label{it:dir-zero} The system $\Rightarrow_{T,\varphi}$ contains $\Rightarrow_\varphi$.
\item \label{it:dir-first} For every pair of terms $t$ and $s$ such that $t \Rightarrow_{T,\varphi} s$, the sequent $\varphi \sststile{}{V} t = s$ is derivable.
\item \label{it:dir-second} For every substitution $\rho$, every term $c$, and every axiom $\psi \sststile{}{V'} t = s$ in $\mathcal{A}_e$
such that sequent $\varphi \sststile{}{V} \psi[\rho] \land c[t[\rho]/x] = c[s[\rho]/x]$ is derivable,
terms $c[t[\rho]/x]$ and $c[s[\rho]/x]$ are equivalent in the system $(\Term_{T,\varphi}^d,\Rightarrow_{T,\varphi})$.
\end{enumerate}
\end{defn}

\begin{remark}[trs-theory]
Often a reduction system on $T$ is defined as $\Rightarrow_{R,\varphi} \cup \Rightarrow_\varphi$,
where $\Rightarrow_{R,\varphi}$ is a restriction of the system $\Rightarrow_R$ to $\Term^d_{T,\varphi}$ for some term rewriting system $R$.
In this case, condition~\eqref{it:dir-zero} is automatically satisfied.
To verify condition~\eqref{it:dir-first}, it is enough to prove that, for every substitution $\rho$ and every pair of terms $t$ and $s$ such that $(t,s) \in R$,
if $\varphi \sststile{T}{V} t[\rho]\!\downarrow$, then $\varphi \sststile{T}{V} t[\rho] = s[\rho]$.
Also, it is enough to prove condition~\eqref{it:dir-second} only for $c = x$ since $a \Rightarrow_{T,\varphi} b$ implies $c[a/x] \Rightarrow_{T,\varphi} c[b/x]$.
\end{remark}

\begin{example}[dir-ax]
Let $T$ be a theory with a set of axioms $\mathcal{A}_d$ and a term rewriting system $R$.
If we define $\mathcal{A}_e$ as the set of axioms of the form $t\!\downarrow\ \sststile{}{\FV(t)} t = s$ for every $(t,s) \in R$,
then $\Rightarrow_R \cup \Rightarrow_\varphi$ is a reduction system on $T \cup \mathcal{A}_e$.
\end{example}

Let $T$ be a theory with separated axioms such that, for every axiom $\psi \sststile{}{V'} t = s$ in $\mathcal{A}_e$, the following conditions hold:
\begin{itemize}
\item The term $t$ is not a variable and $\FV(s) \subseteq \FV(t)$.
\item For every pair $(\varphi,V) \in P_M$ and every substitution $\rho$, if $\varphi \sststile{T}{V} t[\rho]\!\downarrow$, then $\varphi \sststile{T}{V} t[\rho] = s[\rho]$.
\end{itemize}
Then we can define a term rewriting system $R$ as the set of pairs $(t,s)$ such that there is an axiom of the form $\psi \sststile{}{V'} t = s$ in $\mathcal{A}_e$.
The first condition implies that this is indeed a term rewriting system and the second condition implies that it is a reduction system on $T$.
We will say that $T$ \emph{has directed axioms} if these condition hold.
Most of the theories are presented in this way, so we do not need to specify a term rewriting system explicitly.
The theory constructed in \rexample{dir-ax} has directed axioms.

Now let us prove a technical lemma which shows that a sequent $\varphi \sststile{}{V} t = s$ is provable in $T$ if and only if
terms $t$ and $s$ are equivalent in the term rewriting system consisting of the right hand sides of the axioms of $T$ and equalities in $\varphi$.

\begin{lem}[der-eq]
If a sequent $\varphi \sststile{}{V} t = s$ is derivable in a theory $T$, then there exist terms $t_1, \ldots t_n$ such that $t = t_1$, $s = t_n$, and, for every $1 \leq i < n$, $t_i = c[a/x]$ and $t_{i+1} = c[b/x]$
for some terms $a$, $b$, and $c$ such that there is a unique occurrence of the variable $x$ in $c$ and one of the following conditions hold:
\begin{enumerate}
\item There exists an application of \axref{na} in which the premise is derivable from $\varphi$ and the conclusion is either $\varphi \sststile{}{V} a = b$ or $\varphi \sststile{}{V} b = a$.
Moreover, a derivation of $\varphi \sststile{}{V} t = s$ in the natural deduction system contains a derivation of this conclusion as a subderivation.
\item $\varphi = \varphi_1 \land \ldots \land \varphi_k$ and there exists $j$ such that $\varphi_j$ equals to either $a = b$ or $b = a$.
\end{enumerate}
\end{lem}
Moreover, the sequent $\varphi \sststile{}{V} t_i\!\downarrow$ is derivable for every $1 \leq i \leq n$.
\begin{proof}
We prove this by induction on a derivation of $\varphi \sststile{}{V} t = s$ in the natural deduction system.
The rules \axref{nv}, \axref{np}, and \axref{nf} are obvious.
The rules \axref{nh} and \axref{na} follow immediately from assumptions.
We can take $t_1 = a = t$, $t_2 = b = s$, and $c = x$.
Let us consider the rule \axref{ns}.
If $t_1$, \ldots $t_n$ is a sequence for $\varphi \sststile{}{V} s = t$, then we can take the sequence $t_n$, \ldots $t_1$ for $\varphi \sststile{}{V} t = s$.

Finally, let us consider the rule \axref{nl}:
\begin{center}
\AxiomC{$\varphi \sststile{}{V} p = q$}
\AxiomC{$\varphi \sststile{}{V} t'[p/y] = s'[q/y]$}
\RightLabel{\axref{nl}}
\BinaryInfC{$\varphi \sststile{}{V} t'[q/y] = s'[q/y]$}
\DisplayProof
\end{center}
Note that we may assume that there is a unique occurrence of the variable $y$ in $\psi$ since the general rule follows from this special case.
Let $t_1$, \ldots $t_n$ be a sequence for $\varphi \sststile{}{V} p = q$ and let $s_1$, \ldots $s_m$ be a sequence for $\varphi \sststile{}{V} t'[p/y] = s'[p/y]$.
Then $t'[t_n/y]$, \ldots $t'[t_1/y] = s_1$, \ldots $s_m = s'[t_1/y]$, \ldots $s'[t_n/y]$ is a sequence for $\varphi \sststile{}{V} t'[q/y] = s'[q/y]$.

Let us prove that $\varphi \sststile{}{V} t_i\!\downarrow$ is derivable by induction on $i$.
This is true for $i = 1$ by assumption.
Suppose that this is true for some $i$.
Then $t_i = c[a/x]$ and $t_{i+1} = c[b/x]$ and the sequent $\varphi \sststile{}{V} a = b$ is derivable.
Then \axref{nl} implies that the sequent $\varphi \sststile{}{V} t_{i+1}\!\downarrow$ is derivable.
\end{proof}

The following proposition is the main property of theories with reduction systems:

\begin{prop}[conf-main]
Let $T$ be a theory with a reduction system and let $(\varphi,V)$ be a pair in $P_M$.
Then a sequent $\varphi \sststile{}{V} t = s$ is derivable if and only if the terms $t$ and $s$ are equivalent in the system $(\Term_{T,\varphi}^d,\Rightarrow_{T,\varphi})$.
\end{prop}
\begin{proof}
If $t$ and $s$ are equivalent in $(\Term_{T,\varphi}^d,\Rightarrow_{T,\varphi})$, then there is a zig-zag of $\Rightarrow_{T,\varphi}$-reductions between them.
Since the relation $\varphi \sststile{T}{V} - = -$ is an equivalence relation on the set $\Term_{T,\varphi}^d$, we can assume that $t \Rightarrow_{T,\varphi} s$.
Then condition~\eqref{it:dir-first} of \rdefn{directed} implies that $\varphi \sststile{T}{V} t = s$.

If $t$ and $s$ are terms such that $\varphi \sststile{T}{V} t = s$, then \rlem{der-eq} and \rprop{der-separated} imply that
there exists a sequence $t_1$, \ldots $t_n$ of elements of $\Term^d_{T,\varphi}$ such that $t = t_1$, $s = t_n$, and, for every $1 \leq i < n$,
either $t_i \Rightarrow_\varphi t_{i+1}$ or $t_{i+1} \Rightarrow_\varphi t_i$ or there is an axiom $\psi \sststile{}{V'} a = b$ in $\mathcal{A}_e$, a substitution $\rho$, and a term $c$ such that
the sequent $\varphi \sststile{}{V} \psi[\rho]$ is derivable and $t_i = c[a[\rho]/y]$ and $t_{i+1} = c[b[\rho]/y]$ (or vice versa).
Condition~\eqref{it:dir-second} of \rdefn{directed} implies that $t_i$ and $t_{i+1}$ are equivalent in the system $(\Term_{T,\varphi}^d,\Rightarrow_{T,\varphi})$.
\end{proof}

\begin{cor}[conf-main]
Let $T$ be a theory with a reduction system and let $(\varphi,V)$ be a pair in $P_M$.
Then the system $(\Term_{T,\varphi}^d,\Rightarrow_{T,\varphi})$ is confluent if and only if any pair of terms $t$ and $s$ such that $\varphi \sststile{T}{V} t = s$ is joinable in this system.
\end{cor}

Note that if $\Rightarrow_{T,\varphi}$ is defined as $\Rightarrow_R \cup \Rightarrow_\varphi$ for some confluent term rewriting system $R$,
in general this does not imply the confluence of $(\Term_{T,\varphi}^d,\Rightarrow_{R,\varphi})$.
Nevertheless, this is often true under some additional assumptions:

\begin{lem}[morita-conf]
Let $T$ be a theory with a reduction system and let $(\varphi,V)$ be a pair in $P_M$.
Let $R$ be a left-linear term rewriting system.
Suppose that the following conditions hold:
\begin{itemize}
\item The abstract reduction system $(\Term_{T,\varphi}^d,\Rightarrow_R)$ is confluent.
\item For every reduction rule $(t,s) \in R$, if $t$ contains a subterm of the form $e_p(t')$, then $t'$ is not a variable.
\end{itemize}
Then the abstract reduction system $(\Term_{T,\varphi}^d, \Rightarrow_R \cup \Rightarrow_\varphi)$ is confluent.
\end{lem}
\begin{proof}
We can think of variables in $V$ as additional constants.
Then $\Rightarrow_{R,\varphi}$ is the union of two confluent term rewriting systems $R$ and $\Rightarrow_\varphi$.
The last condition implies that they are orthogonal to each other.
It was shown in \cite{raoult} (see also \cite[Theorem~8.6.35]{ohlebusch-advanced}) that the union of confluent orthogonal left-linear systems is confluent.
\end{proof}

\begin{defn}[confluent]
A \emph{confluent} type theory is a type theory $T$ with a reduction system such that equivalent conditions of \rcor{conf-main} hold for every pair $(\varphi,V)$.
\end{defn}

\section{Examples}
\label{sec:examples}

In this section we construct several examples of Morita equivalences and describe other applications of results of this paper.

\subsection{Simple examples}
\label{sec:simple}

In this subsection we consider maps of the form $f : T \to T \cup \mathcal{A}$, where $\mathcal{A}$ is a set of axioms.

\begin{prop}[ext-morita]
Let $\mathcal{A}$ be a set of sequents in a theory $T$.
Suppose that, for every axiom $\psi \sststile{}{V'} \chi$ in $\mathcal{A}$, every pair $(\varphi,V) \in P_M$, and every substitution $\rho$,
the sequent $\varphi \sststile{}{V} \chi[\rho]$ is derivable in $T$ whenever $\varphi \sststile{}{V} \psi[\rho]$ is.
Then, for every pair $(\varphi,V)$ in $P_M$, if a sequent $\varphi \sststile{}{V} \psi$ is derivable in $T \cup \mathcal{A}$, then it is also derivable in $T$.
In particular, the map $T \to T \cup \mathcal{A}$ is a Morita equivalence.
\end{prop}
\begin{proof}
Obvious induction on the derivation of $\varphi \sststile{}{V} \psi$.
\end{proof}

\begin{example}
We already saw examples of such a Morita equivalence in \rprop{der-separated}.
This proposition implies that the map $T \to T \cup \mathcal{A}_d'$ is a Morita equivalence for every theory $T$ with separated axioms.
It also has the following implication.
Suppose that we want to extend $T$ with a typing axiom of the following form:
\[ \psi \sststile{}{V} e_p(\sigma(x_1, \ldots x_k)) = A \]
There are two natural choices for the formula $\psi$: $\varphi_\sigma$ and $\sigma(x_1, \ldots x_k)\!\downarrow$.
Let $T_1 = T \cup \{ \varphi_\sigma \sststile{}{V} e_p(\sigma(x_1, \ldots x_k)) = A \}$ and $T_2 = T \cup \{ \sigma(x_1, \ldots x_k)\!\downarrow\ \sststile{}{V} e_p(\sigma(x_1, \ldots x_k)) = A \}$.
Then the obvious map $T_1 \to T_2$ is a Morita equivalence.
Indeed, if a sequent $\varphi \sststile{}{V} \psi$ is derivable in $T_2$, then it is also derivable in $T_2 \cup \mathcal{A}_d'$.
Since theories $T_1 \cup \mathcal{A}_d'$ and $T_2 \cup \mathcal{A}_d'$ have the same theorems, it is also derivable in $T_1 \cup \mathcal{A}_d'$.
By \rprop{der-separated}, it is also derivable in $T_1$.
\end{example}

\begin{example}
Let $T_\Pi$ be the theory of $\Pi$-types.
One of the axioms (beta reduction) of this theory looks like this:
\[ \ft(B) = A \land \ty(b) = B \land \ty(a) = A \sststile{}{A,B,a,b} \app(A,B,\lambda(A,b),a) = b[a] \]
If we replace this axiom with the following one, then we obtain a new theory which we will denote by $T_\Pi'$.
\[ \app(A,B,\lambda(A,b),a)\!\downarrow\ \sststile{}{A,B,a,b} \app(A,B,\lambda(A,b),a) = b[a] \]

The formula $\app(A,B,\lambda(A,b),a)\!\downarrow$ is equivalent to $\ft(B) = A \land \Pi(A,\ty(b)) = \Pi(A,B) \land \ty(a) = A$.
Thus, $\ft(B) = A \land \ty(b) = B \land \ty(a) = A$ implies it, but not vice versa.
We want to show that the obvious map $T_\Pi \to T_\Pi'$ is a Morita equivalence.
It is a folklore result that the ordinary type theory with $\Pi$-types is confluent.
This does not imply confluence of $T_\Pi'$ immediately since terms of this theory differ from terms of ordinary type theory.
It is possible to prove confluence of $T_\Pi'$, but this proof is beyond the scope of this paper, so we will simply assume that $T_\Pi'$ is confluent.

Let us show that the condition of \rprop{ext-morita} holds.
Suppose that $\varphi \sststile{}{V} \ft(B) = A \land \Pi(A,\ty(b)) = \Pi(A,B) \land \ty(a) = A$ is derivable in $T_\Pi'$.
Since $T_\Pi'$ is confluent and there are no reductions of the form $\Pi(A,B) \Rightarrow t$, it follows that $\ty(b) \Rightarrow B$.
Hence, $\varphi \sststile{}{V} \ty(b) = B$.
This shows that the condition of \rprop{ext-morita} holds.

This also implies that $T_\Pi$ is confluent.
Indeed, the conditions of \rdefn{directed} and \rdefn{confluent} involve only sequents of the form $\varphi \sststile{}{V} \psi$ where $(\varphi,V) \in P_M$
and \rprop{ext-morita} implies that such a sequent is derivable in $T_\Pi$ if and only if it is derivable in $T_\Pi'$.
Moreover, the underlying term rewriting systems of $T_\Pi$ and $T_\Pi'$ coincide.
These facts imply that $T_\Pi$ is confluent if and only if $T_\Pi'$ is confluent.
\end{example}

\begin{example}
We can formulate the beta reduction axiom in one of the following ways:
\begin{align*}
\ft(B) = A \land A = A' \land \ty(b) = B \land \ty(a) = A & \sststile{}{A,A',B,a,b} \app(A,B,\lambda(A',b),a) = b[a] \\
\app(A,B,\lambda(A',b),a)\!\downarrow & \sststile{}{A,A',B,a,b} \app(A,B,\lambda(A',b),a) = b[a]
\end{align*}
Let us denote the theory with the former axiom by $T_\Pi''$ and the theory with the latter by $T_\Pi'''$.
Then we have the following commutative diagram of theories:
\[ \xymatrix{ T_\Pi  \ar[r] \ar[d]  & T_\Pi'' \ar[d] \\
              T_\Pi' \ar[r]         & T_\Pi'''
            } \]
The top arrow is actually an isomorphism and we can prove that the two remaining arrows are Morita equivalences using \rprop{ext-morita} in the same way as we did this for the arrow $T_\Pi \to T_\Pi'$.
Note that even though theories $T_\Pi$ and $T_\Pi''$ are isomorphic they differ as theories with directed axioms.
In particular, the underlying term rewriting systems of $T_\Pi$ and $T_\Pi''$ differ.
The latter is left-linear and this is the main reason why we might be interested in this theory.
\end{example}

We can summarize results of this subsection as follows.
There are several ways to defined a theory of $\Pi$-types, but they are all Morita equivalent.
Also, similar results can be proved for other theories such as the theory of $\Sigma$-types or the theory of identity types.

\subsection{Contractible types}
\label{sec:contr}

The notion of a contractible type was defined by Vladimir Voevodsky \cite{unimath}.
The theory of the contractible type can be formulated in several different ways.
In this subsection we will prove that some of them are Morita equivalent.
The first theory that we will consider is the simplest definition of a contractible type:
\begin{center}
\AxiomC{}
\UnaryInfC{$\Gamma \vdash c_0 : C$}
\DisplayProof
\qquad
\AxiomC{$\Gamma \vdash c : C$}
\UnaryInfC{$\Gamma \vdash \Ceq(c) : \Id(C, c_0, c)$}
\DisplayProof
\end{center}
We will denote this theory by $T_0$.

\begin{prop}
If a theory $T$ is under $T_I$, then $T \to T \amalg T_0$ is a Morita equivalence.
\end{prop}
\begin{proof}
Consider a theory $T'$ which extends $T$ with a type $\vdash C\ \type$, an equivalence $I \vdash e : C$ between $I$ and $C$, a term $\vdash c_0 : C$, and a homotopy $\vdash h : \Id(C, e[\leftI], c_0)$.
Since $C$ is equivalent to $I$ and $I$ is contractible, it follows that $C$ is also contractible.
That is, there is a term $\Ceq'(c) : \Id(c_0, c)$ in $T'$.
Let $T'' = T' \cup \{ \Ceq, h' \}$, where $\Ceq$ satisfies the same axiom as before and $h'$ satisfies the following axiom:
\begin{center}
\AxiomC{$\Gamma \vdash c : C$}
\UnaryInfC{$\Gamma \vdash h'(c) : \Id(\Id(C, c_0, c), \Ceq(c), \Ceq'(c))$}
\DisplayProof
\end{center}

The map $T \to T'$ is a trivial cofibration in the Morita model structure since it is the composition of maps $T \to T \cup \{ C, e \}$ and $T \cup \{ C, e \} \to T \cup \{ C, e, c_0, h \}$ and
these maps are pushouts of the generating trivial cofibrations.
The map $T' \to T''$ is also a trivial cofibration for the same reason.

Finally, let us show that the map $T \to T \amalg T_0$ is a retract of $T \to T''$.
The map $T_0 \to T''$ is defined in the obvious way and the map $f : T'' \to T_0 \amalg T$ is defined in such a way that $T \to T'' \xrightarrow{f} T_0 \amalg T$ equals to $T \to T \amalg T_0$.
Since $C$ and $I$ are both contractible, we can define $f(e)$ simply as $I \vdash c_0 : C$ and $f(h)$ as $\vdash \refl(c_0) : \Id(C, c_0, c_0)$.
Since $C$ is contractible, the type $\Id(\Id(C, c_0, c), \Ceq(c), f(\Ceq'(c)))$ is also contractible, so $f(h')$ can be defined as any inhabitant of this type.
Since $T \to T \amalg T_0$ is a retract of $T \to T''$ and the latter map is a Morita equivalence, it follows that the former map is also a Morita equivalence.
\end{proof}

Next, we will consider the theory of the unit type which we will denote by $T_2$:
\begin{center}
\AxiomC{}
\UnaryInfC{$\Gamma \vdash \unit : \top$}
\DisplayProof
\qquad
\AxiomC{$\Gamma \vdash t : \top$}
\UnaryInfC{$\Gamma \vdash t \deq \unit$}
\DisplayProof
\end{center}
This theory is an extension of the theory of the contractible type by the equation $\Ceq(c) = \refl(c_0)$.
Let us prove that $T \amalg T_0 \to T \amalg T_2$ is a Morita equivalence whenever $T \amalg T_2$ is confluent and satisfies some additional conditions.
First, we need to give several definitions and prove a technical lemma.
Let $t$ be a term of sort $(p,n)$, where $p \in \{ \ty, \tm \}$.
We will say that $t$ is \emph{$\ft$-free} if function symbol $\ft$ does not occur in $t$.
Then we define \emph{the set of contexts} of $t$ as the set of subterms of $t$ of sort $(\ctx,n)$ which are not proper subterms of a subterm of this sort.
In other words, if either $t = x$ or $t = \ty(x)$, then the set of contexts of $t$ is empty,
and if either $t = \sigma_m(\Gamma, t_1, \ldots t_k)$ or $t = \ty(\sigma_m(\Gamma, t_1, \ldots t_k))$, then the set of contexts of $t$ consists of $\Gamma$ and contexts of terms $t_1$, \ldots $t_k$.
We will say that $t$ is \emph{a context-normal form} if $t$ is $\ft$-free and the set of contexts is either empty or a singleton.
In the latter case the single element of the set of contexts of $t$ will be called \emph{the context of $t$}.
If $(\varphi,V)$ is a pair in $P_M$ and $t$ is a term such that $\varphi \sststile{T}{V} t\!\downarrow$, then there is a context-normal form $t'$ such that $\varphi \sststile{T}{V} t = t'$.
This term will be called \emph{the context-normal form of $t$}.

Let $T$ be a theory with a reduction system and let $(\varphi,V)$ be a pair in $P_M$.
Let $\Rightarrow_{T,\varphi}^0$ be an abstract reduction system such that $\Rightarrow_{T,\varphi}$ is the closure of $\Rightarrow_{T,\varphi}^0$ in the sense that $\Rightarrow_{T,\varphi}$ contains $\Rightarrow_{T,\varphi}^0$
and $\sigma(t_1, \ldots t_k) \Rightarrow_{T,\varphi} \sigma(t_1, \ldots t_i', \ldots t_k)$ whenever $t_i \Rightarrow_{T,\varphi} t_i'$ for all function symbols $\sigma_m$ including $v_j$, $\subst$, $\ft$, and $\ty$.
Then we can define another abstract reduction system $\Rightarrow_{T,\varphi}^c$.
Let $t$ be a $\ft$-free term.
Then we define $t \Rightarrow_{T,\varphi}^c s$ in the same way as $t \Rightarrow_{T,\varphi} s$ except for the fact that $\Rightarrow_{T,\varphi}^0$-reductions are not allowed in contexts.
We will say that $T$ \emph{preserves $\ft$-free terms} if, for all terms $t$ and $s$ such that $t \Rightarrow_{T,\varphi}^0 s$, if $t$ is $\ft$-free, then so is $s$.
If $T$ satisfies this condition, then $\Rightarrow_{T,\varphi}^c$ is a relation on $\ft$-free terms.
We will say that $T$ \emph{preserves context-normal forms} if it preserves $\ft$-free terms and, for all $\ft$-free terms $t$ and $s$ such that $t \Rightarrow_{T,\varphi}^0 s$, the set of contexts of $s$ is a subset of the set of contexts of $t$.
If $T$ satisfies it, then $\Rightarrow_{T,\varphi}^c$ is a relation on context-normal forms.

Finally, we need yet another assumption on the theory $T$.
Let $t$ be a $\ft$-free term and let $x_1$, \ldots $x_n$ be variables that occur in contexts of $t$.
Let $r_1$, \ldots $r_n$ and $s$ be terms such that $t[r_1/x_1, \ldots r_n/x_n] \Rightarrow_{T,\varphi}^0 s$.
We will say that $T$ is \emph{context-irrelevant} if, for all such $t$, $r_1$, \ldots $r_n$, and $s$, there exists a term $s'$ such that variables $x_1$, \ldots $x_n$ occur in contexts of $s'$, $s = s'[r_1/x_1, \ldots r_n/x_n]$,
and, for all terms $r_1'$, \ldots $r_n'$, it is true that $t[r_1'/x_1, \ldots r_n'/x_n] \Rightarrow_{T,\varphi}^{0*} s'[r_1'/x_1, \ldots r_n'/x_n]$.
This is a simple technical assumption on $T$ which holds for all theories that occur in practice.

\begin{lem}[context]
Let $T$ be a context-irrelevant theory with a reduction system which preserves context-normal forms and let $(\varphi,V)$ be a pair in $P_M$.
Let $t$ and $s$ be context-normal forms such that either one of them does not have a context or both of them have the same context.
If $t$ and $s$ are joinable under $\Rightarrow_{T,\varphi}$, then they are joinable under $\Rightarrow_{T,\varphi}^c$.
\end{lem}
\begin{proof}
First, let us prove that if $t$ is a $\ft$-free term and $t \Rightarrow_{T,\varphi}^* s$, then $t \Rightarrow_{T,\varphi}^{c*} \Rightarrow_{T,\varphi}^{d*} s$,
where $t_1 \Rightarrow_{T,\varphi}^d t_2$ if $t_1 \Rightarrow_{T,\varphi} t_2$ and not $t_1 \Rightarrow_{T,\varphi}^c t_2$.
To prove this, it is enough to show that if $t \Rightarrow_{T,\varphi}^{d*} \Rightarrow_{T,\varphi}^c s$, then $t \Rightarrow_{T,\varphi}^{c*} \Rightarrow_{T,\varphi}^{d*} s$.
We prove this by induction on the size of $t$ without contexts.
Since $t$ cannot be a variable, we can assume that $t = \sigma_m(\Gamma, t_1, \ldots t_k)$, $t \Rightarrow_{T,\varphi}^{d*} \sigma_m(\Gamma', t_1', \ldots t_k')$, and $\sigma_m(\Gamma', t_1, \ldots t_k') \Rightarrow_{T,\varphi}^c s$.
Then either $s = \sigma_m(\Gamma', t_1', \ldots t_i'', \ldots t_k')$ and $t_i' \Rightarrow_{T,\varphi}^c t_i''$ or $\sigma_m(\Gamma', t_1, \ldots t_k') \Rightarrow_{T,\varphi}^0 s$.
In the former case, we conclude by induction hypothesis.
In the latter case, we use the fact that $T$ is context-irrelevant.
Since $t \Rightarrow_{T,\varphi}^{d*} \sigma_m(\Gamma', t_1', \ldots t_k')$, there exist terms $t'$, $r_1$, \ldots $r_n$, $r_1'$, \ldots $r_n'$ and variables $x_1$, \ldots $x_n$ which occur in contexts of $t$
such that $t = t'[r_1'/x_1, \ldots r_n'/x_n]$ and $\sigma_m(\Gamma', t_1', \ldots t_k') = t'[r_1/x_1, \ldots r_n/x_n]$.
Since $T$ is context-irrelevant, there exists a term $s'$ such that $s = s'[r_1/x_1, \ldots r_n/x_n]$ and $\sigma_m(\Gamma', t_1', \ldots t_k') \Rightarrow_{T,\varphi}^{0*} s'[r_1'/x_1, \ldots r_n'/x_n]$.
Moreover, since $r_i' \Rightarrow_{T,\varphi}^* r_i$ and variables $x_1$, \ldots $x_n$ occur in contexts of $s'$, it is true that $s'[r_1'/x_1, \ldots r_n'/x_n] \Rightarrow_{T,\varphi}^{d*} s'[r_1/x_1, \ldots r_n/x_n]$.
Thus $t \Rightarrow_{T,\varphi}^{0*} s'[r_1'/x_1, \ldots r_n'/x_n] \Rightarrow_{T,\varphi}^{d*} s$.

Now, we can prove the lemma.
If $t \Rightarrow_{T,\varphi}^* q$ and $s \Rightarrow_{T,\varphi}^* q$, then $t \Rightarrow_{T,\varphi}^{c*} q_1 \Rightarrow_{T,\varphi}^{d*} q$ and $s \Rightarrow_{T,\varphi}^{c*} q_2 \Rightarrow_{T,\varphi}^{d*} q$.
Since $T$ preserves context-normal forms, $q_1$ and $q_2$ are context normal forms and the contexts of $q_1$ and $q_2$ coincide with the context of $t$ and $s$.
Reductions $q_1 \Rightarrow_{T,\varphi}^{d*} q$ and $q_2 \Rightarrow_{T,\varphi}^{d*} q$ occur only in contexts of $q_1$ and $q_2$.
This means that the parts of terms $q_1$ and $q_2$ without contexts coincide.
Since the contexts of $q_1$ and $q_2$ are also the same, this implies that $q_1 = q_2$.
\end{proof}

Now, we are ready to formulate the main proposition of this section:
\begin{prop}[contr-main]
Let $T$ be a theory under $\coeT_1 + \sigma + \PathT$ with well-defined function symbols which satisfies the conditions of \rlem{context}.
Suppose that $T \amalg T_2$ is confluent and generated by a term rewriting system as described in \rremark{trs-theory}.
Moreover, assume that, for every reduction rule $(t_1,t_2)$ in this system, if $t_i[s_1, \ldots s_k]$ contains a proper subterm of type $C$, then it is a subterm of $s_j$ for some $j$.
\end{prop}

\begin{remark}
Conditions of \rprop{contr-main} are technical assumptions that hold in many theories that occur in practice.
\end{remark}

\begin{proof}
Clearly, the first condition of \rlem{eq-char-fib} holds for the map $T \amalg T_0 \to T \amalg T_2$.
Thus, to prove that it is a Morita equivalence, it is enough to check that it satisfies the condition of \rlem{char-type-fib}.
That is, we need to prove that, for every pair $(\varphi,V) \in P_M$ and every pair of types $\Gamma \vdash A\ \type$ and $\Gamma \vdash B\ \type$ in $T \amalg T_0$ such that $f(\Gamma) \vdash f(A) \deq f(B)$ is provable in $T \amalg T_0 \cup \varphi$,
there exists a term $\Gamma, A \vdash b : B$ such that $f(\Gamma), f(A) \vdash f(b) \deq v_0$.
It is enough to prove facts \eqref{it:contr-first} and \eqref{it:contr-second} below.
Indeed, since $T \amalg T_2$ is confluent $f(A)$ and $f(B)$ are joinable under $\Rightarrow_{T \amalg T_2, f(\varphi)}^c$ by \rlem{context}.
It follows that they are joinable under relation $\Rightarrow_{T \amalg T_2, f(\varphi)}^{cp}$ (see \eqref{it:contr-first}).
Thus, we can define $b$ as composition of terms $a_1$, $a_2$, and $b'$ defined below.
\begin{enumerate}
\item \label{it:contr-first} There is a confluent subset $\Rightarrow_{T \amalg T_2, f(\varphi)}^{cp}$ of $\Rightarrow_{T \amalg T_2, f(\varphi)}^c$ such that its transitive closure is $\Rightarrow_{T \amalg T_2, f(\varphi)}^c$ and,
for every type $\Gamma \vdash A\ \type$ of $T \amalg T_0$, if $f(A) \Rightarrow_{T \amalg T_2, f(\varphi)}^{cp} A'$, then there exist terms $A''$, $a_1$, and $a_2$ such that $f(A'') = A'$ and the following sequents are derivable:
\begin{align*}
\Gamma, A & \vdash a_1 : A'' \\
\Gamma, A'' & \vdash a_2 : A \\
f(\Gamma), f(A) & \vdash f(a_1) \deq v_0 \\
f(\Gamma), f(A'') & \vdash f(a_2) \deq v_0
\end{align*}
\item \label{it:contr-second} For every pair of types $\Gamma \vdash A'\ \type$ and $\Gamma \vdash B'\ \type$ such that $f(A') = f(B')$, there exists a term $\Gamma, A' \vdash b' : B'$ such that $f(\Gamma), f(A') \vdash f(b') \deq v_0$.
\end{enumerate}

It is easy to see that \eqref{it:contr-second} holds since $f$ is almost injective.
The only function symbols that $f$ identifies are $\Ceq$ and $\refl$.
Formally, for every pair of terms $a$ and $b$ such that $\ctx^n(a)$ and $\ctx^n(b)$ are equivalent for some $n$ and $f(a) = f(b)$, 
we define a term $h_n(a,b)$ such that $\ctx^n(h_n(a,b))$ is equivalent to $(\ctx^n(a), i : I)$, $\leftI^*(h_n(a,b))$ is equivalent to $a$, $\rightI^*(h_n(a,b))$ is equivalent to $b$, and $f(h_n(a,b))$ is a constant homotopy.
Then a term $\Gamma, A \vdash t : B$ can be defined as $\coe_0(h_0(A,B), v_0)$.
Since we have the $\sigma$ rule, $f(t)$ is equivalent to $v_0$.

If $a = b = x$, then $h_n(a,b)$ is the constant homotopy.
If $a = \sigma_m(a_1, \ldots a_k)$ and $b = \sigma_m(b_1, \ldots b_k)$, then $h_n(a,b) = \sigma_{m+1}(h_n(a_1,b_1), \ldots h_n(a_k,b_k))$.
If $a = b = v_i$ and $i < n$, then $h_n(a,b) = v_i$.
If $a = b = v_i$ and $i \geq n$, then $h_n(a,b) = v_{i+1}$.
If $a = \subst_{p,m,k}(a', a_1, \ldots a_k)$ and $b = \subst_{p,m,k}(b', b_1, \ldots b_k)$, then $h_n(a,b) = \subst_{p,m+1,k}(h_k(a',b'), v_n, h_n(a_1,b_1), \ldots h_n(a_k,b_k))$.
The only function symbols that are identified by $f$ are $\Ceq$ and $\refl$.
Thus the remaining case is when $a = \Ceq(c)$ and $b = \refl(c')$.
In this case $c = c' = c_0$ since $\Ceq(c)$ is mapped to $\refl(\unit)$ and the only function symbol which is mapped to $\unit$ is $c_0$.
Let $\Delta = \ctx(h_n(c_0,c_0))$.
It is easy to construct a term $\Delta \vdash p : \Id(\Id(C,c_0,c_0),\Ceq(c_0),\refl(c_0))$ since $C$ is contractible.
Thus we can define $h_n(a,b)$ as $\at(p,i)$.

Now, let us prove \eqref{it:contr-first}.
Let $\Gamma \vdash A\ \type$ be a type of $T \amalg T_0$ such that $f(A) \Rightarrow_{T \amalg T_2, f(\varphi)}^c A'$.
We will prove that there is a type $\Gamma \vdash A''\ \type$ such that $f(A'') = A'$ and a homotopy $\Gamma, i : I \vdash H \type$ between $A$ and $A''$ such that $f(H)$ is the constant homotopy.

For every $\ft$-free term $t$ of $T \amalg T_0$ such that $\ctx^n(t) = \Gamma$, we define terms $g_n(t)$ and $h_n(t)$ such that $\ctx^n(g_n(t)) = \Gamma$ and $\ctx^n(h_n(t)) = (\Gamma, i : I)$
by replacing every subterm $c$ (which is not in a context of $t$) of type $C$ with $c_0$ and $\at(\Ceq(c),i)$, respectively.
Formally, we define $g_n(t)$ and $h_n(t)$ by induction on $t$.
Let us give the definition of $h_n(t)$; $g_n(t)$ is defined similarly.
If $\Gamma, \Delta \vdash t : C$, then $h_n(t) = \Gamma, i : I, h_n(\Delta) \vdash \at(\Ceq(c),i) : C$.
If $t = x$, $t = v_i$, or $t$ has the same sort as $\Gamma$, then $h_n(t)$ is the weakening of $t$.
If $t = \sigma_m(t_1, \ldots t_k)$, then $h_n(t) = \sigma_{m+1}(h_n(t_1), \ldots h_n(t_k))$.
If $t = \subst_{p,m,k}(t, t_1, \ldots t_k)$, then $h_n(t) = \subst_{p,m+1,k+1}(h_k(t), v_n, h_n(t_1), \ldots h_n(t_k))$.

Suppose that $\varphi \sststile{T \amalg T_0}{V} t\!\downarrow$.
Then we can prove that $\varphi \sststile{T \amalg T_0}{V} g_n(t)\!\downarrow \land h_n(t)\!\downarrow$.
To do this, we define a well-founded relation on the set of terms.
We will say that $t[t_1/x_1, ... t_n/x_n] < \sigma_m(t_1, ... t_n)$ if $t_1$, ... $t_n$ are arbitrary terms and $\sigma_m$ is greater than every function symbol that occurs in $t$.
It can be shown that this relation is well-founded.
Note that every term is greater than its subterms.
If $\Gamma, \Delta \vdash t : C$, $t = x$, or $t = v_i$, then $t$ is always defined.
Let us consider the case $t = \sigma_m(t_1, \ldots t_k)$
If $t = \subst_{p,m,k}(t, t_1, \ldots t_k)$, then a similar argument applies, so we omit this case.
We also prove this only for $g_n$, the proof for $h_n$ is similar.

Since $T \amalg T_0$ has well-defined function symbols, there exist terms $A_1, \ldots A_k$ such that all function symbols that occur in $A_i$ are less than $\sigma_m$ and the following condition is satisfied:
\[ \bigwedge_{1 \leq i \leq k} e_{p_i}(x_i) = A_i \ssststile{T}{x_1, \ldots x_k} \sigma_m(x_1, \ldots x_k)\!\downarrow \]
For every such $A_i$ and for all terms $t_1$, \ldots $t_{i-1}$, terms $g_n(A_i[t_1/x_1, \ldots t_{i-1}/x_{i-1}])$ and $A_i[g_n(t_1)/x_1, \ldots g_n(t_{i-1})/x_{i-1}]$ are equivalent.
This is true because $A_i$ cannot have nontrivial terms of type $C$.
Thus, to prove that $\sigma_m(g_n(t_1), \ldots g_n(t_k))$ is defined, it is enough to prove that terms $g_n(A_i[t_1/x_1, \ldots t_{i-1}/x_{i-1}])$ and $g_n(e_{p_i}(t_i))$ are equivalent.
The first term is defined by induction hypothesis.
Since $g_n(e_{p_i}(t_i)) = e_{p_i}(g_n(t_i))$, the second term is also defined by induction hypothesis.
Moreover, terms $A_i[t_1/x_1, \ldots t_{i-1}/x_{i-1}]$ and $e_{p_i}(t_i)$ are equivalent.
Since the theory is confluent, this implies that they reduce to the same term.
\Rlem{context} implies that this is also true for the reduction system $\Rightarrow_{T \amalg T_0, \varphi}^c$

Thus it is enough to prove that $g_n(t) \Rightarrow_{T \amalg T_0, \varphi}^{c*} g_n(s)$ whenever $t \Rightarrow_{T \amalg T_0, \varphi}^c s$.
Moreover, it is enough to prove this only when $t \Rightarrow_{T \amalg T_0, \varphi}^0 s$ since if $t \Rightarrow_{T \amalg T_0, \varphi}^c s$ and the redex is inside $t$,
then this redex is either in a subterm of $t$ of type $C$ or not.
In the former case, $g_n(t) = g_n(s)$.
In the latter case, it is easy to see that $g_n(t) \Rightarrow_{T \amalg T_0, \varphi}^{c*} g_n(s)$.
Assume that $t \Rightarrow_{T \amalg T_0, \varphi}^0 s$.
If the reduction is coming from $T$, then $g_n(t) \Rightarrow_{T \amalg T_0, \varphi}^c g_n(s)$ by the last assumption of \rprop{contr-main}.
The only problem is reduction rules of the form $e_p(x) \Rightarrow_\varphi A$ since terms $A$ may contain subterms of type $C$.
Let us assume that the only subterms of $\varphi$ of type $C$ are $c_0$.

Let $\Rightarrow_{T \amalg T_2, f(\varphi)}^{cp}$ be the subset of $\Rightarrow_{T \amalg T_2, f(\varphi)}^{c*}$ which consists of the same reductions as $\Rightarrow_{T \amalg T_2, f(\varphi)}^c$
except for the reduction $t \Rightarrow_{T \amalg T_2, f(\varphi)} \unit$.
Instead, we let $t \Rightarrow_{T \amalg T_2, f(\varphi)}^{cp} s$ whenever all subterms of $t$ of type $\top$ are replaced with $\unit$ in $s$.
In this case the condition of \eqref{it:contr-first} holds since we can take $A'' = g_n(t)$, $a_2 = \coe^l_0(h_n(t), v_0)$, and $a_1 = \coe^l_2(h_n(t), \rightI, v_0, \leftI)$.
If $t \Rightarrow_{T \amalg T_2, f(\varphi)}^{cp} s$ is an ordinary reduction from $\Rightarrow_{T \amalg T_2, f(\varphi)}^c$, then the condition of \eqref{it:contr-first} holds strictly,
that is there is a term $A''$ such that $t \Rightarrow_{T \amalg T_0, \varphi} A''$.

We proved that $T \amalg T_0 \to T \amalg T_2$ has the lifting property with respect to $(\varphi,V) \in P_M$ if the only subterms of $\varphi$ of type $C$ are $c_0$.
Let us show that this map has the lifting property with respect to all $(\varphi,V) \in P_M$.
Let $\varphi = \bigwedge_{1 \leq i \leq n} e_{p_i}(x_i) = A_i$.
Let $j$ be a number such that, for all $i < j$, the only subterms of $A_i$ of type $C$ are $c_0$.
We prove by induction on $n + 1 - j$ that $T \amalg T_0 \to T \amalg T_2$ has the lifting property.
If $j = n + 1$, then this is true by the proof above.
Suppose that $j \leq n$.
Let $B$ and $b$ be terms such that $\varphi \sststile{T \amalg T_0}{V} B\!\downarrow$ and $f(\varphi) \sststile{T \amalg T_2}{V} e_p(b) = f(B)$.
Let us denote by $\varphi'$ the following formula:
\[ \bigwedge_{1 \leq i < j} e_{p_i}(x_i) = A_i \land \bigwedge_{j \leq i \leq n} e_{p_i}(x_i') = A_i', \]
where $A_j' = g_0(A_j)$ and $A_i' = A_i[\coe^l_0(h_0(A_j),x_j')/x_j]$ for $i > j$.

By induction hypothesis, $T \amalg T_0 \to T \amalg T_2$ has the lifting property with respect to $(\varphi',V')$, where $V' = \{ x_1, \ldots x_{j-1}, x_j', \ldots x_n' \}$.
Let $B' = B[\coe^l_0(h_0(A_j),x_j')/x_j]$.
Since $h_0(A_j)$ is a constant homotopy in $T \amalg T_2$, the sequent $f(\varphi') \sststile{}{V'} f(B') = f(B)$ is derivable in this theory.
It follows that there is a term $b'$ such that $\varphi' \sststile{T \amalg T_0}{V'} e_p(b') = B'$ and $f(\varphi') \sststile{T \amalg T_2}{V'} f(b') = b$.

Now, we define a sequence of terms $a_j$, \ldots $a_n$ such that the following sequent is derivable in $T \amalg T_0$ for all $j \leq k \leq n$:
\[ \bigwedge_{1 \leq i \leq k} e_{p_i}(x_i) = A_i \sststile{T \amalg T_0}{x_1, \ldots x_k} e_{p_k}(a_k) = A_i'[a_j/x_j', \ldots a_{k-1}/x_{k-1}']. \]
Let $a_j = \coe^l_2(h_0(A_j), \rightI, x_j, \leftI)$.
There is an obvious homotopy between $x_j$ and $\coe^l_0(h_0(A_j),a_j)$.
Let us denote it by $H(i)$.
For every $j < k \leq n$, we define a term $a_k'(i)$ as follows:
\[ \coe_1(i. A_k[H(i)/x_j, a_{j+1}'(i)/x_{j+1}, \ldots a_{k-1}'(i)/x_{k-1}], x_k). \]
We define $a_k$ as $a_k'(\rightI)$ for all $j < k \leq n$.

Now, we define a term $b''$ as $b'[a_j/x_j', \ldots a_n/x_n']$.
Then $\varphi \sststile{T \amalg T_0}{V} e_p(b'') = B'[a_j/x_j', \ldots a_n/x_n']$.
Moreover, $f(\varphi) \sststile{T \amalg T_2}{V} f(b'') = b$.
Note that
\[ H' = B[H(i)/x_j, a_{j+1}'(i)/x_{j+1}, \ldots a_n'(i)/x_n] \]
is a homotopy between $B$ and $B'[a_j/x_j', \ldots a_n/x_n']$ which is constant in $T \amalg T_2$.
Thus $\coe^l_2(H',\rightI,b'',\leftI)$ is the required lifting of $b$.
\end{proof}

\section{Conclusion}
\label{sec:conclusion}

We defined several notions of equivalence between type theories: strict Morita equivalence, Morita equivalence, and syntactic equivalence.
It seems that the notion of strict Morita equivalence is rather useless due to the fact that there are no natural nontrivial examples of such an equivalence.
The problem stems from the fact that, to prove that a map is a strict Morita equivalence, we need to work with arbitrary equations between terms, which prevents us from using techniques associated with confluent theories.
Morita equivalence is the main notion of equivalence between theories and syntactic equivalence can be seen as a first approximation to this notion.
It might be easier to check that a map is a syntactic equivalence before tackling the more difficult problem of proving that a map is a Morita equivalence.

The notion of a confluent type theory that we also defined in this paper is very useful when working with theories syntactically.
The reason is that it is very easy to check when a term of a confluent is defined and when two terms are equivalent.
The latter is true if the terms are defined and reduce to the same term.
To check that a term $\sigma(t_1, \ldots t_k)$ of a confluent theory is defined, it is enough to check recursively that subterms $t_1$, \ldots $t_k$ are defined and that equations that define $\sigma$ are satisfied.

We also defined a model structure on categories of type theories with the interval type with Morita equivalences as weak equivalences.
This structure can be used in several ways to prove that a certain map is a Morita equivalence.
First of all, the existence of a model structure implies that Morita equivalences between type theories with the interval type satisfy the 2-out-of-3 property.
We do not know whether this is true for all type theories.
Another useful consequence is that a Morita equivalence which is a cofibrations is always a homotopy equivalence since all objects are fibrant in this model structure.

There is a very explicit description of generating trivial cofibrations in the category of type theories.
This means that a map between arbitrary type theories can always be factored into a cofibration followed by a trivial fibration using the small object argument even if we don't have the model structure
(in the absence of the module structure, the trivial cofibration/fibration factorization is not that useful since we cannot even show that trivial cofibrations are Morita equivalences).
and we believe it is often possible to construct such a factorization explicitly.
In general, a Morita equivalence which is a cofibration may not be a homotopy equivalence, but we believe it is often the case.
For example, this is always true for theories with the interval type since then we have a model structure in which all objects are fibrant.
This leaves us with the question of how to prove that a map is a trivial fibration.
Section~\ref{sec:triv-fib} is devoted to this question.

To prove that a map $f : T_1 \to T_2$ is a trivial fibration, it is necessary to show that, for every pair of terms of $T_1$ which are equivalent in $T_2$, there is a homotopy between them which is constant in $T_2$.
This is usually the most difficult part of a proof that a map is a trivial fibration.
If $T_2$ is confluent, then it is enough to prove conditions \eqref{it:contr-first} and \eqref{it:contr-second} that appear in subsection~\ref{sec:contr}.

There is a problem with this approach.
To describe it, let us assume that $T_1$ and $T_2$ have the same function symbols.
Suppose that we have a reduction $t[a/x] \Rightarrow_{T_2,\varphi} t[b/x]$, where $a$ reduces to $b$ in $T_2$, but not in $T_1$.
To prove \eqref{it:contr-first}, we need to find a term $s$ of $T_1$ such that $f(s) = t[b/x]$.
Since $T_1$ and $T_2$ have the same function symbols, this means that $s$ must be equal to $t[b/x]$.
The problem is that this term may not be defined in $T_1$.
For example, if $t = \sigma(x,a,\refl(a))$, where $\sigma(x,y,p)$ is defined if $p : \Id(A,x,y)$, then $t[a/x]$ is defined in $T_1$, but $t[b/x]$ is not.

We solved this problem in subsection~\ref{sec:contr} by replacing all occurrences of $a$ in the term that we are reducing.
This solution also has a problem.
To check that term $t[b/x]$ is defined, we need to verify that various terms are equivalent, that is reduce to the same term.
The problem is that these terms may contain term $a$ as a subterm \emph{after the reduction}.
This problem does not occur in the case of the unit type since we are replacing all subterms of a certain type
and it is true that if $t \Rightarrow_{T_1,\varphi} s$ and $t$ does not contain subterms of some type, then this is also true for $s$.
This argument does not work if $a \Rightarrow_{T_2,\varphi} b$ is a usual reduction rule such as $\app(A,B,\lambda(A',b),a) \Rightarrow_{T_2,\varphi} b[a]$.

The last problem does not occur if $T_1$ does not have reduction rules apart from typing axioms and the rules for $\subst$.
So we could try to use this argument to give an explicit construction of cofibrant replacement for some theories.
We believe that it is often possible to replace reduction rules of a theory with propositional equalities to get a cofibrant replacement.
The problem is that we cannot assume that theory has the interval type since the theory of the interval type has nontrivial reduction rules.
It is difficult to show that various constructions of a theory with the interval type preserve homotopies, but we believe that it can be done.

Finally, let us describe another idea that can be used to prove that a map $f : T_1 \to T_2$ (between theories with the interval type) is a trivial fibration.
We can modify conditions \eqref{it:contr-first} and \eqref{it:contr-second} slightly so that we do not have to assume that $T_1$ does not have any reduction rules.
Consider a reduction rule of the form $\sigma(a, t_2, \ldots t_k) \Rightarrow_{T_2,\varphi} \sigma(b, t_2, \ldots t_k)$, where $\sigma(a, t_1, \ldots t_k)$ is a term defined in $T_1$.
Then the following term is also defined in $T_1$:
\[ t' = \sigma(b, \coe_0(i.\,A_2[h],t_1), \ldots \coe_0(i.\,A_k[h],t_k)), \]
where $A_1$, \ldots $A_k$ are terms such that $\bigwedge_{1 \leq i \leq k} e_{p_i}(x_i) = A_i \ssststile{T_1}{x_1, \ldots x_k} \sigma(x_1, \ldots x_k)\!\downarrow$ and $i : I \vdash h : A_1$ is a homotopy between $a$ and $b$.
Similar construction can be used in the more general case of the reduction $t[a/x] \Rightarrow_{T_2,\varphi} t[b/x]$.

Terms $t'$ and $t[b/x]$ are equivalent in $T_2$, but they are not equal.
Thus we have to modify condition \eqref{it:contr-first} so that we can construct a term in $T_1$ which is not necessarily equal to $t[b/x]$ in $T_2$, but still is sufficiently close to it.
More specifically, $t'$ may contain additional applications of $\coe_0(i.\,A,-)$ inserted in $t[b/x]$ such that $A$ is a trivial homotopy in $T_2$.
To apply this method, we need to solve two problems.
The first one is that the term in $T_1$ does not match a term in $T_2$ strictly, so it might be more difficult to construct a homotopy in $T_1$ corresponding to a reduction rule of $T_2$.
The second problem is that we need to modify \eqref{it:contr-second} accordingly and it seems that this problem is more serious.

Let us say in conclusion that even though we developed the basic theory of Morita equivalences between type theories it seems that there is much more to be done as we discussed in this section.
Also, it seems that most of this theory is not related to type theories and can be generalized to arbitrary partial Horn theories.

\bibliographystyle{amsplain}
\bibliography{ref}

\end{document}